\documentclass[11pt]{amsart}
\usepackage[top=1.2in, bottom=1.5in, left=1.5in, right=1.5in]{geometry}

\usepackage{amsfonts, amssymb, amsmath, enumerate}

\numberwithin{equation}{section}

\def \N {\mathbb{N}}

\def \R {\mathbb{R}}

\def \Z {\mathbb{Z}}

\def \vol {{\rm vol}}
\def \Id {{\rm Id}}

\def \etc {,\ldots,}

\newtheorem{theorem}{Theorem}[section]
\newtheorem{proposition}[theorem]{Proposition}

\newtheorem{lemma}[theorem]{Lemma}

\theoremstyle{remark}

\begin{document}

\title[Non-central sections of the regular $n$-simplex]{Non-central sections of the regular $n$-simplex}

\author{Hermann K\"onig}
\address[Hermann K\"onig]{Mathematisches Seminar\\
Universit\"at Kiel\\
24098 Kiel, Germany
}
\email{hkoenig@math.uni-kiel.de}

\keywords{Volume, non-central section, simplex}
\subjclass[2000]{Primary: 52A38, 52A40. Secondary: 52A20}

\begin{abstract}
We show that the maximal non-central hyperplane sections of the regular $n$-simplex of side-length $\sqrt 2$ at a fixed distance $t$ to the centroid are those parallel to a face of the simplex, if $\sqrt{\frac {n-2} {3 (n+1)}} < t \le  \sqrt{\frac {n-1} {2 (n+1)}}$ and $n \ge 5$. For $n=4$, the same is true in a slightly smaller range for $t$. This adds to a previous result for $\sqrt{\frac {n-1} {2 (n+1)}} < t \le  \sqrt{\frac n {n+1}}$. For $n=2, 3$, we determine the maximal and the minimal sections for all distances $t$ to the centroid.
\end{abstract}

\maketitle

{\it Dedicated to the memory of Nicole Tomczak-Jaegermann, friend and coauthor}

\section{Introduction}

The Busemann-Petty problem and the hyperplane conjecture of Bourgain initiated an investigation of extremal hyperplane sections of convex bodies. This became a very active area in convex geometry and geometric tomography. \\

K. Ball \cite {B} found the maximal hyperplane section of the $n$-cube in a celebrated paper. The minimal ones had been determined by Hadwiger \cite {Ha} and Hensley \cite {He}. Meyer and Pajor \cite {MP} settled the maximal section problem for the $l_p^n$-balls, if $0 < p <2$, and the minimal one for the $l_1^n$-ball, which was extended by Koldobsky \cite {K} to $l_p^n$-balls for $0 < p < 2$. Important progress in the case $2 < p < \infty$ was made recently by Eskenazis, Nayar and Tkocz \cite {ENT}. Webb \cite {W} found the maximal central sections of the regular $n$-simplex. The paper by Nayar and Tkocz \cite {NT} gives an excellent survey on extremal sections of classical convex bodies. \\

For non-central sections at distance $t$ to the centroid, not too many results are known. Moody, Stone, Zach and Zvavitch \cite {MSZZ} proved that the maximal sections of the $n$-cube at very large distances from the origin are those perpendicular to the main diagonal. Pournin \cite {P} showed the same result for slightly smaller distances. Liu and Tkocz \cite {LT} proved a corresponding result for the $l_1^n$-ball, if the distance $t$ to the origin satisfies $\frac 1 {\sqrt 2} < t \le 1$. K\"onig \cite {Ko2} extended this to $\frac 1 {\sqrt 3} < t \le \frac 1 {\sqrt 2}$. Maximal non-central sections of the regular $n$-simplex of side-length $\sqrt 2$ were studied by K\"onig \cite{Ko} for relatively large distances $t$ to the centroid. In this paper, we extend this result to smaller distances $t$ for $n \ge 4$, resulting in the range $\sqrt{\frac {n-2} {3 (n+1)}} < t \le  \sqrt{\frac n {n+1}}$. Note that for $t > \sqrt{\frac n {n+1}}$, the intersection is empty. The maximal sections are those parallel to faces, i.e. those perpendicular to the diagonals from the centroid to the vertices of the simplex, in dimension $n=4$ in a slightly smaller range for $t$. The proof uses a volume formula of Dirksen \cite {D} which extended a result of Webb \cite {W}. This formula involves alternating terms with removable singularities, if
$\sqrt{\frac {n-2} {3 (n+1)}} < t \le  \sqrt{\frac {n-1} {2 (n+1)}}$, resulting in difficulties to find extrema of section volumes. \\

For small dimensions $n=2, 3$, we determine the maximal and the minimal hyperplane sections for all possible distances $|t| \le  \sqrt{\frac n {n+1}}$. \\

To formulate our results precisely, we introduce some definitions and notations. Let $\N := \{ \ n \in \Z \ | \ n \ge 1 \ \}$ denote the positive integers. For the regular $n$-simplex $K= \Delta^n$ we use its representation in $\R^{n+1}$
$$\Delta^n := \{ x \in \R^{n+1} \ | \ x = (x_j)_{j=1}^{n+1} , \ x_j \ge 0 , \ \sum_{j=1}^{n+1} x_j =1 \} . $$
Let $a \in S^n := \{ \ x = (x_j)_{j=1}^{n+1} \in \R^{n+1} \ | \  \sum_{j=1}^{n+1} x_j^2 =1 \ \}$ be a direction vector and $t \in \R$. We will always assume that $\sum_{j=1}^{n+1} a_j = 0$ so that the centroid $c = \frac 1 {n+1} (1, \cdots , 1) \in \Delta^n$ is in the hyperplane $a^\perp$. By
$H(a,t) := \{ x \in \Delta^n \ | \ \langle a, x \rangle = t \}  =  \{ t a\} + a^\perp $ we denote the hyperplane orthogonal to $a$ at distance $t$ to the centroid. Thus $c \in H_0(a)$. Then
$$A(a,t) := \vol_{n-1}(H(a,t) \cap \Delta^n) = \vol_{n-1}(\{ x \in \Delta^n \ | \ \langle a, x \rangle = t \}) $$
is the {\it parallel section function} of $\Delta^n$. We aim to find the extrema of $A( \cdot,t)$ on $S^n$ for fixed $t$. For $t \ne 0$ we have non-central sections. \\

Note that $\sqrt{ \frac n {n+1} }$ is the distance of the vertices $e_j=(0 \etc 0, \underbrace{1}_j , 0 \etc 0)$ of $\Delta^n$ to the centroid $c$; it is the maximum coordinates of $a \in S^n$ with $\sum_{j=1}^{n+1} a_j = 0$ can attain, and in this case all other coordinates are equal to $-\frac 1 {\sqrt{n(n+1)}}$. The side-length of the simplex $\Delta^n$ is $\sqrt 2$, its height $\sqrt {\frac{n+1} n}$ and its volume $\frac{\sqrt{n+1}}{n!}$. The distance of the centroid to the midpoint of edges of $\Delta^n$ is $\sqrt{\frac {n-1} {2 (n+1)}}$. By symmetry, we may assume that $a = (a_j)_{j=1}^{n+1} \in S^n$ satisfies $a_1 \ge a_2 \cdots \ge a_{n+1}$. Throughout this paper, we make the following \\

{\bf General Assumption.} A direction vector $a = (a_j)_{j=1}^{n+1} \in \R^n$ is always assumed to satisfy $\sum_{j=1}^{n+1} a_j^2 = 1$,
$\sum_{j=1}^{n+1} a_j = 0$ and $a_1 \ge a_2 \cdots  \ge a_{n+1}$. \\

Note that under this assumption always $a_{n+1} < 0$. We introduce the sets $H_+(a,t) = \{ \ x \in \Delta^n \ | \ \langle a, x \rangle > t \ \}$ and $H_-(a,t) = \{ \ x \in \Delta^n \ | \ \langle a, x \rangle < t \ \}$. For $\sqrt{\frac {n-1} {2 (n+1)}} < t < \sqrt {\frac n {n+1}}$ the set $H_+(a,t)$ contains just one vertex, namely $e_1$. Further, $\sqrt{\frac {n-k+1} {k (n+1)}}$ is the distance from the centroid of $\Delta^n$ to the centroid of a $(k-1)$-boundary simplex. For $\sqrt{\frac {n-k} {(k+1) (n+1)}} < t < \sqrt {\frac {n-k+1} {k (n+1)}}$ the set $H_+(a,t)$ contains at most $k$ vertices of $\Delta^n$. This follows from Lemma \ref{lem1} below. If $a_1 < t$ and $x \in \Delta^n$, $t > \sum_{j=1}^{n+1} a_j x_j = \langle a, x \rangle$, hence $H(a,t) = \emptyset$. If $a_1 = t$ and $x \in \Delta^n$, $t >  \sum_{j=1}^{n+1} a_j x_j = \langle a, x \rangle$ as well, unless $x = e_1$. Then $H(a,t) = \{e_1\}$ and $A(a,t) = 0$. Therefore, if $A(a,t) > 0$, the General Assumption implies that $a_1 > t$. \\

The unit vectors in the direction from the centroid of $\Delta^n$ to the centroid $\frac{e_1 + \cdots + e_k} k$ of a (particular) boundary $(k-1)$-simplex are
$$a^{(k)} := \sqrt{\frac {k(n-k+1)} {n+1}} (\underbrace{ \frac 1 k \etc \frac 1 k}_k, \underbrace{-\frac 1 {n-k+1} \etc -\frac 1 {n-k+1} }_{n-k+1} )  \in S^n \ ,$$
$k = 1, \cdots , n$. Thus $a^{(1)}$ points from the centroid of $\Delta^n$ to a vertex, $a^{(2)}$ to the midpoint of an edge and $a^{(3)}$ to the center of a boundary triangle. A hyperplane orthogonal to the "main diagonal" $a^{(1)}$ is parallel to a boundary $(n-1)$-face of $\Delta^n$. Note that $a^{(n-k+1)}$ is essentially equal to $-a^{(k)}$, up to a permutation of coordinates. Since $H_t(a) = H_{-t}(-a)$, we have
\begin{equation}\label{eq1.1}
A(a,t) = A(-a,-t) \quad , \quad A(a^{(k)},t) = A(a^{(n-k+1)},-t) \ , \ k=1 \etc n \ .
\end{equation}

We now formulate our main results.

\begin{theorem}\label{th1}
Let $n \ge 5$, $\Delta^n$ be the regular $n$-simplex and assume that $t \in \R$ satisfies $\sqrt{\frac {n-2} {3 (n+1)}} <  t \le  \sqrt{\frac n {n+1}}$. Let $a \in S^n$ with $\sum_{j=1}^{n+1} a_j = 0$. Then $H_t(a^{(1)})$ is a maximal hyperplane section of $\Delta^n$ at distance $t$ to the centroid, i.e.
$$A(a,t) \le A(a^{(1)},t) = \frac {\sqrt{n+1}}{(n-1)!} \ \left(\frac n {n+1} \right)^{n/2} \ \left(\sqrt{ \frac n {n+1}} -t \right)^{n-1} \ . $$
All maximal sections at distance $t$ are parallel to a boundary $(n-1)$-face, and the orthogonal vector is a permutation of $a^{(1)}$. For $n=4$, the same is true if $0.3877 \simeq \tilde{t} < t < \sqrt{ \frac 4 5 }$. For $n=4$ and $\sqrt{\frac {n-2} {3 (n+1)}} = \sqrt{\frac 2 {15}} \le t < \tilde{t}$, $a^{(2)}$ yields a maximal hyperplane at distance $t$ to the centroid. We have $A(a^{(2)},\tilde{t}) = A(a^{(1)},\tilde{t})$.
\end{theorem}

The range $\sqrt{\frac {n-1} {2 (n+1)}} <  t \le  \sqrt{\frac n {n+1}}$ was already covered by a result in K\"onig \cite {Ko}. As in the results of Moody, Stone, Zach and Zvavitch \cite {MSZZ} for the $n$-cube and of Liu and Tkocz \cite {LT} and K\"onig \cite {Ko2} for the $l_1^n$-ball, the maximal hyperplanes far from the barycenter are those perpendicular to the main diagonal. By K\"onig \cite {Ko}, $a^{(1)}$ is a local maximum of $A(\cdot,t)$ if $\frac{2n+1}{n(n+2)} \sqrt{\frac n {n+1}} < t$. It can be routinely verified that $\frac{2n+1}{n(n+2)} \sqrt{\frac n {n+1}} \le \sqrt{\frac {n-2} {3 (n+1)}}$ for all $n \ge 4$. The value $\tilde{t}$ is the solution of a cubic equation. \\

For $t=0$, by Webb \cite {W} the maximal central sections of $\Delta^n$ are those containing $(n-1)$ vertices and the midpoint of the remaining two vertices, orthogonal to vectors like $\frac 1 {\sqrt 2} (1,0 \etc 0,-1)$. Filliman \cite {F} stated without proof that the minimal central sections are those parallel to a face of $\Delta^n$. Recently, Tang \cite{T} proved an asymptotically sharp lower bound for the volume of central sections of $\Delta^n$. \\
For dimensions $n=2, 3$, we give the extrema of $A(\cdot,t)$ for all $t$: \\

\begin{proposition}\label{prop2}
For $n=2$ and $0 \le t \le \sqrt{\frac 2 3}$, let
$$a^{[t]} := \left(\frac 1 2 (\sqrt{2-3 t^2} - t), t , -\frac 1 2 (t+\sqrt{2-3 t^2}) \right) \in S^2 \ \text{  and}  $$
$$a^{\{t\}} := \left(\tilde{a_1},b_+,b_- \right) \in S^2, \tilde{a_1} :=t + \sqrt{t^2-\frac 1 6},  b_{\pm} := -\frac {\tilde{a_1}} 2 \pm \frac 1 2 \sqrt{2 - 3 \tilde{a_1}^2} \ . $$
Then the extrema of $A(\cdot,t)$ on $S^2$ are given by
$$ Maximum :  \left\{\begin{array}{c@{\quad}l}
a^{[t]} \quad \ ,  \; 0 \le t \le \frac 1 {{\sqrt 6}} : & A(a^{[t]},t) = \sqrt 3 \frac 1 {\sqrt{2-3 t^2}} \\
a^{\{t\}} \ , \frac 1 {{\sqrt 6}} < t \le \frac 5 4 \frac 1 {{\sqrt 6}} : & A(a^{\{t\}},t) = \frac 1 {2 \sqrt 3} \frac 1 {t+\sqrt{t^2-1/6}} \\
a^{(1)} \ ,  \frac 5 4 \frac 1 {{\sqrt 6}} \le t < \sqrt{\frac 2 3} : & A(a^{(1)},t) = \frac {2 \sqrt 2} 3 - \frac 2 {\sqrt 3} t
\end{array}\right\}  \; , \; $$
$$ Minimum : \left\{\begin{array}{c@{\quad}l}
a^{(1)} \; , \; 0 \le t \le \frac 1 {{\sqrt 6}} : &  A(a^{(1)},t) = \frac {2 \sqrt 2} 3 - \frac 2 {\sqrt 3} t
\end{array}\right\}  \ . $$
\end{proposition}

Partially, this can be found already in Dirksen \cite {D} and K\"onig \cite {Ko}. For $0 \le t \le \frac 1 {\sqrt 6}$, $a^{[t]}$ turns continuously from $a^{[0]} = \frac 1 {\sqrt 2} (1,0,-1)$ as given by Webb's result \cite {W} to $a^{(2)} = \frac 1 {\sqrt 6} (1,1,-2)$, and for $\frac 1 {{\sqrt 6}} < t \le \frac 5 4 \frac 1 {{\sqrt 6}}$, $a^{\{t\}}$ turns continuously from $a^{(2)}=\frac 1 {\sqrt 6} (1,1,-2)$ to $a^{(1)} = \frac 1 {\sqrt 6} (2,-1,-1)$. Up to a permutation of coordinates, $a^{(2)}$ is just $-a^{(1)}$. The distance of the center to a side of the triangle $\Delta^2$ is $t = \frac 1 {\sqrt 6}$. Enlarging $t$ a bit, turning the side slightly and moving it partially outside of $\Delta^2$, intersects $\Delta^2$ in about $\frac 1 2$ of its length. Thus there is a discontinuity of $A(a^{\max}, \cdot)$ at $t=\frac 1 {\sqrt 6}$, where $a^{\max}$ denotes the direction of largest value of the parallel section function,
$$A(a^{(2)},\frac 1 {\sqrt 6}) = \sqrt 2  > \lim_{t \searrow 1/{\sqrt 6}} A(a^{\{t\}},t) = \frac 1 {\sqrt 2} \ . $$

\begin{proposition}\label{prop3}
For $n=3$ and $0 \le t \le \frac {\sqrt 3} 2$, let $t_0 := \frac {9+4\sqrt{16-6 \sqrt 3 }}{2(3 \sqrt 3 + 16)} \simeq 0.43575$, $t_1 := \frac{\sqrt 6 +1}{10} \simeq 0.34495$ and
$$a^{[t]} := \left(\frac 1 2 \sqrt{2-8 t^2} - t, t , t, -(\frac 1 2 \sqrt{2-8 t^2}+t) \right) \in S^3 \ \text{  and}  $$
$$a^{\{t\}} := (\tilde{a_1},\tilde{a_1},b_+,b_-) \in S^3, b_{\pm} := -\tilde{a_1} \pm \frac 1 2 \sqrt{2-8 \tilde{a_1}^2} \ , $$
where $t \in [\frac 1 {2 \sqrt 3} , t_1]$ and $\tilde{a_1} \in [\frac 1 {2 \sqrt 3} , \frac 1 2]$ is the unique solution of  $\phi(a_1) = t$,
$$ \phi(a_1) := \frac{ a_1(5-12 a_1^2) + (6 a_1^2 - \frac 1 2) \sqrt{28 a_1^2-1} } {1+36 a_1^2} \ .$$
Then the extrema of $A(\cdot,t)$ on $S^3$ are given by
$$ Maximum :  \left\{\begin{array}{c@{\quad}l}
a^{[t]} \; \ ,  \; 0 \le t \le \frac 1 { 2 {\sqrt 3}} : & A(a^{[t]},t) = \frac 1 {\sqrt{2-8 t^2}} \\
a^{\{t\}} \ , \frac 1 { 2 {\sqrt 3}} < t \le t_1  &  \\
a^{(2)} \;  , \; \  t_1 \le t \le t_0 : & A(a^{(2)},t) = \frac 1 2 - 2 t^2 \\
a^{(1)} \; ,  t_0 \le t < \frac{\sqrt 3} 2 : & A(a^{(1)},t) = \frac {3 \sqrt 3} 8  (\frac {\sqrt 3} 2 -t)^2
\end{array}\right\}  \; , \; $$
$$ Minimum : \left\{\begin{array}{c@{\quad}l}
a^{(1)} \; , \; 0 \le t \le \frac 1 {2 {\sqrt 3}} : &  A(a^{(1)},t) = \frac {3 \sqrt 3} 8  (\frac {\sqrt 3} 2 -t)^2
\end{array}\right\}  \ . $$
\end{proposition}

For $\frac 1 {2 \sqrt 3} < t \le t_1$, $A(a^{\{t\}},t) = \psi(\tilde{a_1})$, $\psi(a_1) := \frac{4 \left[ a_1(5+52 a_1^2)- (2a_1^2+ \frac 1 2)\sqrt{28 a_1^2-1} \right] }{(1+36 a_1^2)^2}$. Again, $a^{[t]}$ turns continuously from  $a^{[0]} = \frac 1 {\sqrt 2} (1,0,0,-1)$ as given by Webb \cite {W} to $a^{(3)} = \frac 1 {2 \sqrt 3} (1,1,1,-3)$ which is $-a^{(1)}$ up to a permutation of coordinates. For $\frac 1 {2 \sqrt 3} < t \le t_1$, $a^{\{t\}}$ turns continuously from $a^{(3)}$ to $a^{(2)} = \frac 1 2  (1,1,-1,-1)$. Again there is a discontinuity at the distance of the center to a boundary triangle $t = \frac 1 {2 \sqrt 3}$,
$$A(a^{(3)},\frac 1 {2 \sqrt 3}) = \frac {\sqrt 3} 2 > \lim_{t \searrow 1/{2 \sqrt 3}} A(a^{\{t\}},t) = \frac 5 9 \frac {\sqrt 3} 2 = \frac 5 {18} {\sqrt 3} \ . $$
The last value stems from the fact that for $t \searrow \frac 1 {2 \sqrt 3}$, also $\tilde{a_1} \searrow \frac 1 {2 \sqrt 3}$. The factor $\frac 5 9$ is easily explained: Dividing a boundary triangle of $\Delta^3$ by a line through its center parallel to a side into a triangle and a trapezoid , the area of the trapezoid is $\frac 5 9$-th of the area of the triangle itself. \\

In the next chapter we state a general formula for $A(a,t)$ which involves alternating terms with removable singularities and give some consequences, e.g. a formula for $A(a^{(k)},t)$. After some preparations in chapter 3, we prove Theorem \ref{th1} in chapter 4. The proof uses some ideas also found in K\"onig \cite {Ko2}, though there are essential differences due to the non-symmetry of $\Delta^n$ and the additional constraint $\sum_{j=1}^{n+1} a_j = 0$. Basically, we show that there are at most three different coordinates in critical points $a \in S^{n-1}$ of $A(\cdot,t)$ in addition to the largest two coordinates, which is then reduced to two with predetermined multiplicity. This leads to specific functions of the largest two coordinates $(a_1,a_2)$ of $a$ to be investigated. We then apply some monotonicity result proved in chapter 3. In chapter 5 we verify Propositions \ref{prop2} and \ref{prop3}. \\

The author would like to thank the referee for carefully evaluating the paper, for his questions and suggestions, helping to improve the manuscript. \\

\section{A volume formula for sections of the $n$-simplex}

The general assumption on $a$ implies

\begin{lemma}\label{lem1}
Let $a \in S^n$, $\sum_{j=1}^{n+1} a_j=0$, $a_1 \ge a_2 \ge \cdots \ge a_{n+1}$ and $1 \le k \le n$. \\
i) If $a_1 \ge \cdots \ge a_k \ge 0$,  $a_k \le \sqrt{\frac{n-k+1}{k(n+1)}}$. In particular, $a_1 \le \sqrt{\frac n {n+1}}$. If $a_1 > a_2$, $a_3 < \sqrt{\frac{n-2}{3(n+1)}}$ \ . \\
ii) If $a_1 \ge a_2 \ge 0$, $a_1 \le \psi(a_2) := \sqrt{\frac{n-1} n} \sqrt{1 - \frac{n+1} n a_2^2} - \frac {a_2} n$. $\psi$ is a decreasing function with $\psi^2 = Id$, $a_1 = \psi(a_2)$ if and only if $a_2 = \psi(a_1)$.
\end{lemma}

Therefore, if $\sqrt{\frac{n-k}{(k+1)(n+1)}} < t \le \sqrt{\frac{n-k+1}{k(n+1)}}$, as assumed in Theorem \ref{th1} for $k=2$ or $k=1$, $a_{k+1} < t$ and there are $\le k$ vertices $e_j$ of $\Delta^n$ in $H_+(a,t)$. \\
We will use the function $\psi$ throughout the paper. \\

\begin{proof}
i) Using the Cauchy Schwarz inequality, we find that \\
$ \sum_{j=1}^k a_j = - \sum_{j=k+1}^{n+1} a_j \le \sqrt{n-k+1} \sqrt{\sum_{j=k+1}^{n+1} a_j^2}  = \sqrt{n-k+1} \sqrt{1- \sum_{j=1}^k a_j^2}$,
which together with $k a_k \le \sum_{j=1}^k a_j$ and $k a_k^2 \le \sum_{j=1}^k a_j^2$ implies $a_k \le \sqrt{\frac{n-k+1}{k(n+1)}}$. \\

ii) For $k=2$, the previous inequalities yield $(a_1+a_2)^2 \le (n-1) ( 1- (a_1^2+a_2^2) )$. This quadratic inequality implies
$a_1 \le \sqrt{\frac{n-1} n} \sqrt{1 - \frac{n+1} n a_2^2} - \frac {a_2} n =: \psi(a_2)$. Since
$\psi'(a_2) = -\sqrt{\frac{n-1} n} \frac{(n+1)/n a}{\sqrt{1-(n+1)/n a^2}} - \frac 1 n < 0$, $\psi$ is strictly decreasing. We have $a_1 = \psi(a_2)$ if and only if $(a_1+a_2)^2 = (n-1) ( 1- (a_1^2+a_2^2) )$, which is symmetric in $a_1$ and $a_2$.
\end{proof}

We need an explicit formula for $A(a,t)$. Put $x_+ := \max(x,0)$ for any $x \in \R$. Then:

\begin{proposition}\label{prop2.1}
Let $n \in \N$, $t \in \R$ with $|t|\le \sqrt{\frac n {n+1}}$ and $a \in S^n$ with $\sum_{j=1}^{n+1} a_j =0$. Assume that $a_1 > a_2 > \cdots > a_{n+1}$. Then
\begin{equation}\label{eq2.1}
A(a,t) = \frac{\sqrt{n+1}}{(n-1)!} \sum_{j=1}^{n+1} \frac{(a_j-t)_+^{n-1}}{\prod_{k=1,k \ne j}^{n+1} (a_j - a_k)} \ .
\end{equation}
\end{proposition}

This is Corollary 2.4 of Dirksen \cite {D}. It is stated there for sequences $a \in S^n$ without $\sum_{j=1}^{n+1} a_j =0$, but explained how to transform it into the above formula. Formula \eqref{eq2.1} generalizes a result of Webb \cite {W} for central sections. The proof uses the Fourier transform technique explained in Koldobsky's book \cite {K} leading to $A(a,t) = \frac 1 {2 \pi} \int_\R \prod_{k=1}^{n+1} \frac 1 {1+ i (a_k - t) s } \ ds$, which is evaluated using the residue theorem. \\

By \eqref{eq1.1}, $A(a,t) = A(-a,-t)$. Possibly exchanging $(a,t)$ with $(-a,-t)$, we may apply \eqref{eq2.1} with at most $[\frac{n+1}2]$ non-zero terms, thus with less than $n$ non-zero terms. For the standard vectors $a^{(k)}$ we have the following result.

\begin{proposition}\label{prop2.2}
Let $1 \le k \le n$ and $-\sqrt{\frac k {(n-k+1)(n+1)}} < t < \sqrt{\frac{n-k+1}{k(n+1)}}$. Then
$A(a^{(k)},t) =  \frac{\sqrt{n+1}}{(n-1)!} \binom{n-1}{k-1} \sqrt{\frac{k(n-k+1)}{n+1}}^n \left(t+\sqrt{\frac k {(n-k+1)(n+1)}} \right)^{k-1} \left(\sqrt{\frac{n-k+1}{k(n+1)}}-t \right)^{n-k}$.
\end{proposition}

Since $|\langle a^{(k)},e_j \rangle| \ge \sqrt{\frac k {(n-k+1)(n+1)}}$ for all $j=1 \etc n+1$, $A(a^{(k)},t) = 0$ for $t < -\sqrt{\frac k {(n-k+1)(n+1)}}$. \\

\begin{proof}
By continuity, formula \eqref{eq2.1} also holds for all sequences $a \in S^n$ with $a_1 > a_2 > \cdots > a_k > t \ge a_{k+1}, \cdots , a_{n+1}$ where some or all of the $a_l$ with $l >k$ may be equal. For $k=1$, equation \eqref{eq2.1} immediately implies the claim. For $a_1 > a_2 > t \ge a_3 \etc a_{n+1}$,
$$A(a,t) = \frac{\sqrt{n+1}}{(n-1)!} \frac 1 {a_1-a_2} \left( \frac{(a_1-t)^{n-1}}{\prod_{j=3}^{n+1} (a_1-a_j)} - \frac{(a_2-t)^{n-1}}{\prod_{j=3}^{n+1} (a_2-a_j)} \right) \ . $$
For $k=2$, with $a^{(2)} = (b,b,c, \cdots , c)$, $b := \sqrt{\frac{n-1}{2(n+1)}}$, $c := -\sqrt{\frac 2 {(n-1)(n+1)}}$, let
$a(\varepsilon):= \frac 1 {\sqrt{1-\varepsilon^2}} (b+ \varepsilon, b-\varepsilon, c, \cdots , c)$, $\sum_{j=1}^{n+1} a(\varepsilon)_j=0$, $\sum_{j=1}^{n+1} a(\varepsilon)_j^2 = 1$,  $a(\varepsilon) \to a^{(2)}$ and $A(a(\varepsilon),t) \to A(a^{(2)},t)$ as $\varepsilon \to 0$. The continuous extension of \eqref{eq2.1} to $a(\varepsilon)$ then contains two alternating non-zero summands with $b+\varepsilon = a_1$ and $b-\varepsilon =a_2$, provided that $b-\varepsilon > t$. We find for $t <  \sqrt{\frac{n-1}{2(n+1)}}$
$$\frac{(n-1)!}{\sqrt{n+1}} A(a^{(2)},t) = \lim_{\varepsilon \to 0} \frac 1 {2 \varepsilon} \left( \frac{(b-t+\varepsilon)^{n-1}} {(b-c+\varepsilon)^{n-1}} - \frac{(b-t-\varepsilon)^{n-1}} {(b-c-\varepsilon)^{n-1}} \right)  = g'(b) \ , $$
where $g(x):= \frac{(x-t)^{n-1}}{(x-c)^{n-1}}$, $g'(x) = (n-1) (t-c) \frac{(x-t)^{n-2}}{(x-c)^n}$, \\
$g'(b) = (n-1) \left(t+ \sqrt{\frac 2 {(n-1)(n+1)}} \right) \left(\sqrt{\frac{2(n-1)}{n+1}} \right)^n \left(\sqrt{\frac{n-1}{2(n+1)}}-t \right)^{n-2} \ , $
since $\frac 1 {b-c} = \sqrt{\frac {2(n-1)} {n+1}}$. \\
For $k>2$ and $a_k = (\underbrace{b \etc b}_k , \underbrace{c \etc c}_{n+1-k})$, $b := \sqrt{\frac{n-k+1}{k(n+1)}}$, $c:= -\sqrt{\frac k {(n-k+1)(n+1)}}$ and
$t < b$ we consider $k$ equidistant points of step-size $\varepsilon$ around $b$, $b + j \varepsilon$, $j= - [\frac k 2] \etc [\frac k 2]$, to define $a(\varepsilon)$ similarly, and then \eqref{eq2.1} yields a $(k-1)$-st order $\varepsilon$-step size difference $\Delta^{(k-1)}_{\varepsilon} g$ of the function
$g(x) := \frac{(x-t)^{n-1}}{(x-c)^{n-k+1}}$, namely
$$\frac{(n-1)!}{\sqrt{n+1}} A(a^{(k)},t) = \lim_{\varepsilon \to 0} \frac 1 {(k-1)!} \frac 1 {\varepsilon^{k-1}} \Delta^{(k-1)}_{\varepsilon} g(b) = \frac 1 {(k-1)!} g^{(k-1)}(b) \ . $$
The power in the denominator of $g$ is reduced by the fact that there are $(k-1)$ products of differences of values near $b$ resulting in the $(k-1)$-st order difference. One verifies by induction that
$$\frac 1 {(k-1)!} g^{(k-1)}(x) = \binom{n-1}{k-1} (t-c)^{k-1} \frac {(x-t)^{n-k}}{(x-c)^n} \ , $$
and with $\frac 1 {b-c} = \sqrt {\frac {k(n-k+1)}{n+1}}$ we have that \\
$A(a^{(k)},t) =  \frac{\sqrt{n+1}}{(n-1)!} \binom{n-1}{k-1} \sqrt{\frac{k(n-k+1)}{n+1}}^n \left(t+\sqrt{\frac k {(n-k+1)(n+1)}} \right)^{k-1} \left(\sqrt{\frac{n-k+1}{k(n+1)}}-t \right)^{n-k} \ . $
\end{proof}

In particular,
$A(a^{(1)},t) = \frac{\sqrt{n+1}}{(n-1)!} \left(\sqrt{\frac n {n+1}} \right)^n \left(\sqrt{\frac n {n+1}} - t \right)^{n-1} \ , \ t < \sqrt{\frac n {n+1}} \ , $ \\
$A(a^{(2)},t) = \frac{\sqrt{n+1}}{(n-2)!} \left(\sqrt{\frac {2(n-1)} {n+1}} \right)^n \left(\sqrt{\frac 2 {(n-1)(n+1)}}+t \right) \left(\sqrt{\frac {n-1} {2(n+1)}} - t \right)^{n-2} \ , \ t < \sqrt{\frac {n-1} {2(n+1)}} \ . $ \\

For all $\frac c {\sqrt{n(n+1)}} < t < \sqrt{\frac n {n+1}}$, $A(a^{(1)},t) > A(a^{(2)},t)$ is true asymptotically for $c \simeq 2.6363$:
$A\left(a^{(1)},\frac c {\sqrt{n(n+1)}} \right) - A \left(a^{(2)},\frac c {\sqrt{n(n+1)}} \right) = O(\frac 1 n)$ leads to the requirement $2(\sqrt 2 + c) = \exp(1+(\sqrt 2 -1) c)$, which has only one positive solution $c \simeq 2.6363$. Let $t_n$ be the positive solution in dimension $n$ of $A(a^{(1)},t_n) = A(a^{(2)},t_n)$ less than $\sqrt{\frac{n-1}{2(n+1)}}$. Then $t_3 \simeq 0.4357$, $t_4 \simeq 0.3877$ and $t_5 \simeq 0.3426$.

\section{Prerequisites}

In this chapter we prove two results which are needed in the proof of Theorem \ref{th1} in the next chapter. The first one is used to reduce the number of different coordinates of extremal critical points of $A(\cdot,t)$, the second is a monotonicity result for $A(a,t)$.

\begin{proposition}\label{prop3.1}
Let $n \ge 4$, $\sqrt{\frac{n-2}{3(n+1)}} < t < a_2 < a_1$, $p, r, s \in \N$ and $c, d, e \in \R$ be such that $a_1+a_2+c+d+e = 0$, $a_1^2+a_2^2+\frac{c^2} p + \frac{d^2} r + \frac{e^2} s =1$, $p+r+s=n-1$. Then
$$\left(\frac{a_1-t}{a_2-t} \right)^{n-1} > \left(\frac{a_1-\frac c p}{a_2- \frac c p} \right)^{p+1} \left(\frac{a_1- \frac d r}{a_2- \frac d r} \right)^{r+1} \left(\frac{a_1- \frac e s}{a_2- \frac e s} \right)^{s+1} \ . $$
\end{proposition}

For the proof we need the

\begin{lemma}\label{lem}
Let $n \ge 4$, $\sqrt{\frac{n-2}{3(n+1)}} < t < a_2 < a_1$ and $c \in \R$. Then: \\
(i) If $c \le 0$, $\left(\frac{a_1-t}{a_2-t} \right) > \left(\frac{a_1-c}{a_2-c} \right)^3$. \\
(ii) If $c \le g(n):= 2 \sqrt{\frac{n-2}{3(n+1)}} - \sqrt{\frac{n-1}{2(n+1)}}$, $\left(\frac{a_1-t}{a_2-t} \right) > \left(\frac{a_1-c}{a_2-c} \right)^2$.
\end{lemma}

\begin{proof}
(i) Since $\left(\frac{a_1-c}{a_2-c} \right)$ is decreasing in $c$, it suffices to prove (i) for $c=0$. This means $\left(\frac{a_1-t}{a_2-t} \right) > \left(\frac{a_1}{a_2} \right)^3$ or equivalently $t (a_1^3-a_2^3) -a_1 a_2 (a_1^2-a_2^2) > 0$, which is equivalent to $t > \frac{a_1a_2(a_1+a_2)}{(a_1-a_2)^2+3a_1 a_2}$. By Lemma \ref{lem1} $a_1 \le \psi(a_2)$, hence
$a_1+a_2 \le \psi(a_2)+a_2 = \sqrt{\frac{n-1} n} \sqrt{1-\frac{n+1} n a_2^2} + \frac{n-1} n a_2 =:\gamma(a_2)$,
$\gamma'(a_2) = - \sqrt{\frac{n-1} n} \frac{n+1} n \frac{a_2} {\sqrt{1-\frac{n+1} n a_2^2}} + \frac{n-1} n \ge 0$, since $\frac{n-1} n \ge \frac{(\frac{n+1} n)^2 a_2^2}{1-\frac{n+1} n a_2^2}$ is satisfied in view of $a_2 \le \sqrt{\frac{n-1}{2(n+1)}}$ which holds by Lemma \ref{lem1}. Thus $\gamma$ is increasing in $a_2$, so that
$a_1+a_2 \le \gamma(\sqrt{\frac{n-1}{2(n+1)}}) = \sqrt{\frac{2(n-1)}{n+1}}$. Therefore for $n \ge 4$
$$\frac{a_1 a_2 (a_1+a_2)}{(a_1-a_2)^2+3 a_1 a_2} < \frac{a_1+a_2} 3 \le \frac 1 3 \sqrt{\frac{2(n-1)}{n+1}} \le \sqrt{\frac{n-2}{3(n+1)}} < t $$
is satisfied and (i) holds for $c \le 0$. \\

(ii) In view of (i), we may assume that $c>0$. The claim $\left(\frac{a_1-t}{a_2-t} \right) > \left(\frac{a_1-c}{a_2-c} \right)^2$ is equivalent to $(a_1-a_2)[(a_1+a_2-2 c) t +c^2- a_1 a_2] > 0$. Therefore we need $t > \frac{a_1a_2-c^2}{a_1+a_2-2c} =: h(c)$. Since $h'(c) = 2 \frac{(a_1-c)(a_2-c)}{(a_1+a_2-2c)^2} > 0$, if $c < a_2 < a_1$, $h$ is increasing in $[0,a_2]$. Note here that $c \le g(n) < t < a_2$. Now $c_0 = t - \sqrt{(a_1-t)(a_2-t)}$ is the solution of
$t = h(c) = \frac{a_1a_2-c^2}{(a_1+a_2-2c)}$ with $c_0 < t < a_2$. Thus for $c < c_0$, $h(c) < t$. By the arithmetic-geometric mean inequality $\sqrt{(a_1-t)(a_2-t)} \le \frac{a_1+a_2} 2 - t \le \sqrt{\frac{n-1}{2(n+1)}} - t$ and hence $c_0 = t - \sqrt{(a_1-t)(a_2-t)} \ge 2 t - \sqrt{\frac{n-1}{2(n+1)}} > 2 \sqrt{\frac{n-2}{3(n+1)}} - \sqrt{\frac{n-1}{2(n+1)}} = g(n)$. Thus, if $c \le g(n)$, $c < c_0$ and $h(c) < t$ is satisfied.
\end{proof}

\vspace{0,5cm}

{\it Proof} of Proposition \ref{prop3.1}. \\

Assume that $\frac c p \le \frac d r \le \frac e s$. Then $c < 0$. By assumption $\sqrt{\frac{n-2}{3(n+1)}} < a_2 < a_1$ and by Lemma \ref{lem1} i), $\frac e s \le \sqrt{\frac{n-2}{3(n+1)}}$. We distinguish some cases for $n \ge 4$: \\

a) If $\frac d r , \frac e s > g(n)$,
$\left(\frac{a_1-t}{a_2-t} \right) > \left(\frac{a_1- \frac c p}{a_2- \frac c p} \right)^3$, $\left(\frac{a_1-t}{a_2-t} \right) > \max \left( \left(\frac{a_1- \frac d r}{a_2- \frac d r} \right) , \left(\frac{a_1- \frac e s}{a_2- \frac e s} \right) \right)$ and the claim will follow from $\frac{p+1} 3 + r+1 +s+1 \le n-1 = p+r+s$, i.e. $p \ge 4$. This is satisfied: Since $c < 0$, $|c|=a_1+a_2+d+e \ge 2 \sqrt{\frac{n-2}{3(n+1)}} +(r+s) g(n) = 2(1+r+s) \sqrt{\frac{n-2}{3(n+1)}} - (r+s) \sqrt{\frac{n-1}{2(n+1)}}$ and $\frac{c^2} p = 1-a_1^2-a_2^2- r \left(\frac d r \right)^2 - s \left(\frac e s \right)^2 \le 1 - \frac 2 3 \frac{n-2}{n+1} - (r+s) g(n)^2$, implying
\begin{equation}\label{eq3.1}
p \ge \frac{ \left(2(1+r+s) \sqrt{\frac{n-2}{3(n+1)}} - (r+s) \sqrt{\frac{n-1}{2(n+1)}} \right)^2 }{1 - \frac 2 3 \frac{n-2}{n+1} - (r+s) g(n)^2} \ .
\end{equation}
Since $r, s \in \N$, $r+s \ge 2$. For $n=5$ this yields $p \ge 3.035$, i.e. $p \ge 4$. The right side of \eqref{eq3.1} is increasing in $r+s$ and in $n$. Thus $p \ge 4$ holds for all $n \ge 5$. For $n=4$, $p+r+s=3$ means $p=r=s=1$. But \eqref{eq3.1} yields $p>1$. Thus this case is impossible for $n=4$. \\

b) If $\frac d r \le g(n) < \frac e s$, we have $\left(\frac{a_1-t}{a_2-t} \right) > \left(\frac{a_1- \frac d r}{a_2- \frac d r} \right)^2$ and the claim will follow from
$\frac{p+1} 3 + \frac{r+1} 2 + s+1 \le n-1 = p+r+s$, i.e. $11 \le 4 p + 3 r$. This requires $p \ge 2$ or $p=1$ and $r \ge 3$. We check this similarly as above, if $d \ge 0$: Then $|c| \ge a_1+a_2+d+e \ge 2 \sqrt{\frac{n-2}{3(n+1)}} +s \ g(n)$, $\frac{c^2} p \le 1 - \frac 2 3 \frac{n-2}{n+1} - s \ g(n)^2$. This yields \eqref{eq3.1} with $r+s$ replaced by $s$. For $s=1$, $n=4$, \eqref{eq3.1} yields $p \ge 1.19$, i.e. $p \ge 2$ is satisfied, as needed. For $s>1$ or $n>4$, $p$ will be even larger. For $n=4$, this case is impossible, since $p+r+s=n-1=3$ requires $p=r=s=1$.  \\
If $d < 0$, we only need for the claim that $\frac{p+1} 3 + \frac{r+1} 3 + s+1 \le n-1 = p+r+s$, which means $p+r \ge 3$. We show that $p=r=1$ is impossible so that $p+r \ge 3$ will follow. If $p=r=1$, $0 = a_1+a_2+c+d+e \ge 2 \sqrt{\frac{n-2}{3(n+1)}} + s \ g(n) +c +d \ge 2(1+s) \sqrt{\frac{n-2}{3(n+1)}} -s \sqrt{\frac{n-1}{2(n+1)}} - \sqrt{ \frac 2 3 \frac{n+7}{n+1} } =: \gamma(s,n)$, using $c^2+d^2 \le 1 -a_1^2-a_2^2 \le \frac{n+7}{3(n+1)}$, $|c+d| \le \sqrt{ \frac 2 3 \frac{n+7}{n+1} }$. For $s=1$, $\gamma(s,n) >0$, if $n \ge 6$. For $n=5$, we need $4=p+r+s=2+s$, i.e. $s=2$. However, also $\gamma(2,5)>0$. Thus $p+r \ge 3$ is satisfied. The case $n=4$ needs a separate argument. \\

c) If $\frac d r, \frac e s \le g(n)$, we have that $\left(\frac{a_1-t}{a_2-t} \right) > \max\left( \left(\frac{a_1- \frac d r}{a_2- \frac d r} \right)^2, \left(\frac{a_1- \frac e s}{a_2- \frac e s} \right)^2 \right)$, and we need
$\frac{p+1} 3 + \frac{r+1} 2 + \frac{s+1} 2 \le n-1 = p+r+s$, i.e. $8 \le 4 p + 3 (r+s)$, which is always satisfied, also for $n=4$. \\

d) For $n=4$, case b) with $d<0$ was left open. Then \\
$d, e = -\frac{a_1+a_2+c} 2 \pm \frac 1 2 \sqrt{ 2 - 3(a_1^2+a_2^2+c^2) -2(a_1a_2+a_1c+a_2c) }$. Calculation shows
\begin{align*}
\left(\frac{a_1-d}{a_2-d} \right) \left(\frac{a_1-e}{a_2-e} \right) &= 1 + (a_1-a_2) \frac{2(a_1+a_2)+c}{(a_1+a_2+c)^2 + 2a_2^2 - a_1 c - \frac 1 2} \ , \\
\left(\frac{a_1-t}{a_2-t} \right) &= 1 + (a_1-a_2) \frac 1 {a_2-t} \ .
\end{align*}
We will prove that $\left(\frac{a_1-t}{a_2-t} \right) > \left(\frac{a_1-d}{a_2-d} \right) \left(\frac{a_1-e}{a_2-e} \right)$. The claim of Proposition \ref{prop3.1} will then follow from the  inequality $\frac 2 3 +2 < 3$ for the exponents. We have to show that
$\phi(a_1,a_2) := \frac{(a_1+a_2+c)^2 + 2 a_2^2 - a_1 c - \frac 1 2}{2(a_1+a_2)+c} > a_2 - t$. Using $a_1^2+a_2^2 \ge \frac 4 {15}$, $|c| \le \sqrt{\frac{11}{15}}$, it can be checked that $\frac{\partial \phi}{\partial a_1} = \frac{1+2(a_1^2-a_2^2)+4a_1a_2+2a_1 c -c^2}{(2(a_1+a_2)+c)^2} \ge 0$. Therefore
$\phi(a_1,a_2) \ge \phi(a_2,a_2) = \frac{(2 a_2+c)^2+2a_2^2-a_2c- \frac 1 2}{4 a_2+c} =: k(c)$. We want to minimize $k(c)$ on the interval $c \in [-1,0]$. We have $k'(c) = \frac{6 a_2^2+8 a_2 c +c^2+ \frac 1 2}{(4 a_2+c)^2} = 0$ for
$c_{\pm} = -4a_2 \pm \frac 1 2 \sqrt{40a_2^2-2}$. Only $c_+$ satisfies $c_+ \ge -1$ as needed for $a \in S^4$, and $c_+$ is the minimum of $k$. We have $k(c_+)= \sqrt{40 a_2^2-2} - 5 a_2$, hence
$\phi(a_1,a_2) \ge  \sqrt{40 a_2^2-2} - 5 a_2$. It is easy to see that this is $> a_2 - \sqrt{\frac 2 {15}} > a_2 - t$.  \hfill $\Box$

\vspace{0,5cm}

The second auxiliary result concerns the monotonicity of a very special function appearing in the proof of Theorem \ref{th1} in the situation that $(n-2)$ of the coordinates of a critical point $a \in S^n$ of $A(\cdot,t)$ with fixed $a_1 > a_2 > t$ coincide.

\begin{proposition}\label{prop3.2}
Let $n \ge 4$, $\sqrt{\frac{n-2}{3(n+1)}} \le t < a_2 < a_1$. Let
$$W = W(a_1,a_2) := \sqrt{(n-1)(1-a_1^2-a_2^2)-(a_1+a_2)^2} \ , $$
$$F = F(a_1,a_2) := \frac{(a_1-t)^{n-1}} { \left(a_1+\frac{a_1+a_2}{n-1} + \frac 1 {\sqrt{n-2} (n-1)} W \right)^{n-2} \left(a_1+\frac{a_1+a_2}{n-1} - \frac {\sqrt{n-2}}{n-1} W \right) } \ , $$
$$G = G(a_1,a_2) := \frac{(a_2-t)^{n-1}} { \left(a_2+\frac{a_1+a_2}{n-1} + \frac 1 {\sqrt{n-2} (n-1)} W \right)^{n-2} \left(a_2+\frac{a_1+a_2}{n-1} - \frac {\sqrt{n-2}}{n-1} W \right) } \ , $$
and $f = f(a_1,a_2) := \frac{F(a_1,a_2) - G(a_1,a_2)}{a_1-a_2}$. Then for fixed $a_1$, $f$ is strictly increasing in $a_2$.
\end{proposition}

By Lemma \ref{lem1} $a_2 < \psi(a_1)$, and this guarantees that $W$ is well-defined. For $a_2 = \psi(a_1)$, $W =W(a_1,\psi(a_1)) = 0$. $W$ decreases if $a_2$ increases. For the proof we need

\begin{lemma}\label{lem2}
Let $n \ge 4$, $\sqrt{\frac{n-2}{3(n+1)}} \le t < a_2 < a_1$ and \\
$W(a_1,a_2) := \sqrt{(n-1)(1-a_1^2-a_2^2)-(a_1+a_2)^2}$, $c := -\frac{a_1+a_2}{n-1} - \frac 1 {\sqrt{n-2} (n-1)} W(a_1,a_2)$ and $d := -\frac{a_1+a_2}{n-1} + \frac{\sqrt{n-2}}{n-1} W(a_1,a_2)$. Then : \\
(a) $a_2 - d \ge 3 (a_2-t) + (a_1-a_2)$ , \\
(b) $a_2 - c \ge 5 (a_2-t) + \frac{a_1-a_2}{n-1}$, \\
(c) $\frac{a_1-c}{a_2-c} < \frac 3 2$.
\end{lemma}

\begin{proof}
(a) Let $\delta :=a_2-t$, $\varepsilon := a_1-t$. We claim that $t-d \ge \delta + \varepsilon$. Inserting $\delta, \varepsilon$ into $t-d$, this is equivalent to
$t-d = \frac{n+1}{n-1} t + \frac{\delta + \varepsilon}{n-1} - \frac{\sqrt{n-2}}{n-1} W(a_1,a_2) \ge \delta + \varepsilon$ which means
$\sqrt{n-2} \  W(a_1,a_2) \le (n+1) t - (n-2) (\delta + \varepsilon)$. This is equivalent to
\begin{align*}
(n-2) & \left[ n-1 - 2 (n+1) (t^2 + (\delta + \varepsilon) t) - n (\delta^2 + \varepsilon^2) - 2 \delta  \varepsilon \right] \\
& \le (n+1)^2 t^2 - 2 (n-2) (n+1) (\delta + \varepsilon) t +(n-2)^2 (\delta + \varepsilon)^2 \ .
\end{align*}
The terms with $(\delta + \varepsilon) t$ on both sides cancel, and the inequality is clearly satisfied if
$(n-2)(n-1) \le t^2 ((n+1)^2 + 2 (n-2) (n+1)) = t^2 \ 3 (n+1) (n-1)$, i.e. $t^2 \ge \frac{n-2}{3(n+1)}$, which holds by assumption on $t$. Hence
$t-d \ge a_1+a_2-2 t = 2 (a_2-t) + (a_1-a_2)$, $a_2-d = a_2-t + t-d \ge 3 (a_2-t) + (a_1-a_2)$. \\

(b) We have $t-c = t + \frac{a_1+a_2}{n-1} + \frac 1 {\sqrt{n-2} (n-1)} W(a_1,a_2) \ge t + \frac{a_1+a_2}{n-1} = t + \frac 2 {n-1} a_2 + \frac {a_1-a_2}{n-1}$.
We claim that $ t + \frac 2 {n-1} a_2 \ge 4 (a_2-t)$ holds. This is equivalent to $t \ge (\frac 4 5 - \frac 2 5 \frac 1 {n-1}) a_2$. Since
$t \ge \sqrt{\frac{n-2}{3(n+1)}} = \sqrt{\frac 2 3 \frac{n-2}{n-1} } \sqrt{\frac{n-1}{2(n+1)}} \ge \sqrt{\frac 2 3 \frac{n-2}{n-1} } \ a_2$,
$\sqrt{\frac 2 3 \frac{n-2}{n-1}} \ge \frac 4 5 - \frac 2 5 \frac 1 {n-1}$ will imply the claim, and this is true for all $n \ge 4$. Hence
$a_2 - c = a_2 - t + t - c \ge 5 (a_2-t) + \frac{a_1-a_2}{n-1}$. \\

(c) By Lemma \ref{lem1} $\frac{a_1}{a_2} \le \frac{\psi(a_2)}{a_2} = \sqrt{\frac{n-1} n} \sqrt{\frac 1 {a_2^2} - \frac{n+1} n} - \frac 1 n$. Using $a_2 \ge \sqrt{\frac{n-2}{3(n+1)}}$, we find $y:= \frac{a_1}{a_2} \le \frac{n+1} n \sqrt{\frac{n-1}{n-2}} \sqrt 2 - \frac 1 n$. Since $\frac{a_1+x}{a_2+x}$ is decreasing for $x \ge 0$ and $-c \ge 0$, $-c \ge \frac{a_1+a_2}{n-1}$, we conclude $\frac{a_1-c}{a_2-c} \le \frac{a_1+\frac{a_1+a_2}{n-1}}
{a_2+\frac{a_1+a_2}{n-1}} = \frac {ny+1}{n+y}$. We claim that $\frac {ny+1}{n+y} < \frac 3 2$. This is equivalent to $y < \frac{3n-2}{2n-3}$. Using the above estimate for $y$, it suffices to show $\sqrt{\frac{2(n-1)}{n-2}} < \frac n {n+1} \left(\frac{3n-2}{2n-3} + \frac 1 n \right) = 3 \frac{n-1}{2n-3}$, or equivalently, $2(2n-3)^2 < 9 (n-1)(n-2)$, which holds for all $n \ge 4$. Hence $\frac{a_1-c}{a_2-c} < \frac 3 2$ for all $n \ge 4$.
\end{proof}

{\it Remark.} By Lemma \ref {lem2} (a), $a_2+ \frac{a_1+a_2}{n-1} - \frac{\sqrt{n-2}}{n-1} W = a_2-d > 0$, so that the denominators of $F(a_1,a_2)$ and $G(a_1,a_2)$ in Proposition \ref{prop3.2} are positive and $F(a_1,a_2) >0$, $G(a_1,a_2) >0$. \\

\vspace{0,5cm}

{\it Proof} of Proposition \ref{prop3.2}. \\
We denote by $'$ the derivative with respect to $a_2$. Then
$$f'(a_1,a_2) = \frac{ F(a_1,a_2) ( 1 + (a_1-a_2) (\ln F)'(a_1,a_2) ) - G(a_1,a_2) ( 1 + (a_1-a_2) (\ln G)'(a_1,a_2) ) } {(a_1-a_2)^2} \ . $$
We will show that $f'(a_1,a_2) \ge 0$. This is equivalent to
$$\frac{F(a_1,a_2)}{G(a_1,a_2)} \left[ 1 + (a_1-a_2) (\ln F)'(a_1,a_2) \right] \ge 1 + (a_1-a_2) (\ln G)'(a_1,a_2)  \ . $$
We have $W'(a_1,a_2) = - \frac{n a_2 + a_1}{W(a_1,a_2)}$. Let $u := a_1 + \frac{a_1+a_2}{n-1}$, $v := a_2 + \frac{a_1+a_2}{n-1}$. Then
$$(\ln F)'(a_1,a_2) = - \left(\frac{ \frac{n-2}{n-1}-\frac{\sqrt{n-2}}{n-1} \frac{na_2+a_1} {W(a_1,a_2)} } {u+\frac 1 {\sqrt{n-2} (n-1)} W(a_1,a_2)} +
\frac{ \frac 1{n-1} + \frac{\sqrt{n-2}}{n-1} \frac{na_2+a_1} {W(a_1,a_2)} } {u - \frac {\sqrt{n-2}}{n-1} W(a_1,a_2)} \right) \ . $$
A tedious calculation with the common denominator
$$h(u):= \left(u + \frac 1 {\sqrt{n-2} (n-1)} W(a_1,a_2) \right) \left(u - \frac {\sqrt{n-2}}{n-1} W(a_1,a_2) \right)$$
shows that this is equal to $(\ln F)'(a_1,a_2) = - \frac{h(u)-h(v)}{(a_1-a_2) h(u)}$. Therefore
$$1 + (a_1-a_2) (\ln F)'(a_1,a_2) = 1 - \frac{h(u)-h(v)}{h(u)} = \frac{h(v)}{h(u)} = \left(\frac{a_2-c}{a_1-c} \right) \left(\frac{a_2-d}{a_1-d} \right) \ , $$
where $c := -\frac{a_1+a_2}{n-1} - \frac 1 {\sqrt{n-2} (n-1)} W(a_1,a_2)$, $d := -\frac{a_1+a_2}{n-1} + \frac{\sqrt{n-2}}{n-1} W(a_1,a_2)$. We have
$a_1+a_2+(n-2) c + d = 0$, $a_1^2+a_2^2+(n-2) c^2 + d^2 =1$. We get
$$\frac{F(a_1,a_2)}{G(a_1,a_2)} \left[ 1 + (a_1-a_2) (\ln F)'(a_1,a_2) \right] = \left(\frac{a_1-t}{a_2-t} \right)^{n-1} \left(\frac{a_2-c}{a_1-c} \right)^{n-1} \left(\frac{a_2-d}{a_1-d} \right)^2 $$
and we have to show that this is $\ge 1 + (a_1-a_2) (\ln G)'(a_1,a_2)$. We have
$$(\ln G)'(a_1,a_2) = \frac{n-1}{a_2-t} - \frac{(n-2) \frac n {n-1} - \frac{\sqrt{n-2}}{n-1} \frac{n a_2 +a_1}{W(a_1,a_2)} } {y + \frac 1 {\sqrt{n-2} (n-1)} W(a_1,a_2)} - \frac{ \frac n {n-1} + \frac{\sqrt{n-2}}{n-1} \frac{n a_2 +a_1}{W(a_1,a_2)} } {y - \frac{\sqrt{n-2}}{n-1} W(a_1,a_2)} \ . $$
Using that $c - d = - \frac 1 {\sqrt{n-2}} W(a_1,a_2)$, $a_2-c = y + \frac 1 {\sqrt{n-2} (n-1)} W(a_1,a_2)$ and $a_2-d = y - \frac{\sqrt{n-2}} {n-1} W(a_1,a_2)$, we find after some calculation
$$(\ln G)'(a_1,a_2) = (n-1) \left(\frac 1 {a_2-t} - \frac 1 {a_2-c} \right) - \frac 2 {a_2-d} \ . $$
We have to show
\begin{align}\label{eq3.2}
\left(\frac{a_1-t}{a_2-t} \right)^{n-1} & > \left(\frac{a_1-c}{a_2-c} \right)^{n-1} \left(\frac{a_1-d}{a_2-d} \right)^2 \ \times \\
 & \times \left(1+ (a_1-a_2) \left[(n-1) (\frac 1 {a_2-t} - \frac 1 {a_2-c} ) - \frac 2 {a_2-d} \right] \right) \nonumber
\end{align}
Let $x := \frac{a_1-a_2}{a_2-t} > 0$, $y := \frac{a_1-a_2}{a_2-c} >0$ and $z := \frac{a_1-a_2}{a_2-d} >0$. Then we need
\begin{equation}\label{eq3.3}
(1+x)^{n-1} > (1+y)^{n-1} (1+z)^2 \left(1+ (n-1)(x-y) -2z \right) \ .
\end{equation}
By Lemma \ref{lem2}, $y \le \frac x 5$ and $z \le \frac {a_1-a_2}{3(a_2-t) + (a_1-a_2)} = \frac x {3+x}$. Now the right side of \eqref{eq3.3} is increasing in $z$ and $y$, since their derivatives satisfy
$$2 (1+z) \left[1 + (n-1)(x-y)-2 z -(1+z) \right] \ge 2(1+z) \left[\frac 4 5 (n-1) x - \frac {3 x}{3+x} \right] > 0 \ , $$
\begin{align*}
(n-1)(1+y)^{n-2} & \left[1+(n-1)(x-y)-2z -(1+y) \right] \\
& \ge (n-1) (1+y)^{n-2} \left[(n-1) \frac 4 5 x - \frac {2x}{3+x} - \frac x 5 \right] > 0 \ .
\end{align*}
Therefore \eqref{eq3.3} will be satisfied if for all $x > 0$
\begin{equation}\label{eq3.4}
\phi_n(x):= (1+x)^{n-1}- \left(1+\frac x 5 \right)^{n-1} \left(1+ \frac x {3+x} \right)^2 \left(1+ (n-1) \frac 4 5 x - \frac {2 x}{3+x} \right) > 0
\end{equation}
holds. We use induction on $n$ to show that $\phi_n(x) >0$ for all $n \ge 5$ and all $0 < x \le \frac 5 2$. For $n=5$, expansion shows that
$\phi_5(x) = \frac{x^2}{(1+\frac x 3)^3} \sum_{j=0}^6 \alpha_j x^j$, $\alpha_0 = \frac 7 5$, $\alpha_j \ge 0$ for $j=1, \cdots , 5$,
$\alpha_6 = - \frac {64}{84375} < 0$. Since $\alpha_6 (\frac 5 2)^6 = \frac 5 {27} < \frac 1 5$, $\phi_5(x) > \frac 6 5 \frac{x^2}{(1+\frac x 3)^3} > 0$. For $n \ge 5$, assume that $\phi_n(x) >0$ for all $0 < x \le \frac 5 2$. Then
\begin{align*}
\phi_{n+1}(x) - \phi_n(x) & = x (1+x)^{n-1} - x \left(1+\frac x 5 \right)^{n-1} \left(1+ \frac x {3+x} \right)^2 \left(1+ \frac 4 {25} n x - \frac 2 5 \frac x {3+x} \right) \\
 & > x \phi_n(x) + x \left(1+\frac x 5 \right)^{n-1} \left(1+\frac x {3+x} \right)^2 \left(\frac {16}{25} n x - \frac 4 3 x \right) > 0 \ .
\end{align*}
For $x \ge \frac 5 2$, $\frac x 5 \ge \frac 1 2$, we use that by Lemma \ref{lem2} (c), $\frac{a_1-c}{a_2-c} = 1+y < \frac 3 2$, $y < \frac 1 2$. Then \eqref{eq3.3} will be satisfied if
$$\Phi_n(x) := \left(\frac 2 3 (1+x) \right)^{n-1} - \left(1+ \frac x {3+x} \right)^2 \left(1+ (n-1) (x- \frac 1 2) - \frac {2 x}{3+x} \right) > 0$$
holds. For $n=5$, $\Phi_5(x) = \frac{\sum_{j=0}^7 \beta_j (x-\frac 5 2)^j}{(1+\frac x 3)^3}$ with $\beta_j > 0$ for all $j=0, \cdots , 7$, $\beta_0 > 75$. Thus $\Phi_5(x) > 0$ for all $x \ge \frac 5 2$. For $n \ge 5$ and $x \ge \frac 5 2$,
\begin{align*}
\Phi_{n+1}(x) \ge \frac 7 3 \left(\frac 2 3 (1+x) \right)^{n-1} - & \left(1+ \frac x {3+x} \right)^2 \left(1+ (n-1) (x- \frac 1 2) - \frac {2 x}{3+x} \right)  \\
& - \left(1+\frac x {3+x} \right)^2 \left(x-\frac 1 2 \right) \ .
\end{align*}
Thus $\Phi_n(x) > 0$ implies $\Phi_{n+1}(x) >0$, since for $x \ge \frac 5 2$ and $n \ge 5$ we have
$\frac 4 3 \left(\frac 2 3 (1+x) \right)^{n-1} > \frac 1 3 (1+x)^3 > \frac 1 3 (1+x)^2 (x- \frac 1 2) > (1+ \frac x 3)^2 (x- \frac 1 2) > \left(1+\frac x {3+x} \right)^2 \left(x-\frac 1 2 \right)$. Therefore we have verified \eqref{eq3.3} and \eqref{eq3.2} for all $n \ge 5$. \\
For $n=4$, we use that by Lemma \ref{lem2} we actually have a better estimate for $y$, namely $y \le \frac{a_1-a_2}{5(a_2-t) + \frac{a_1-a_2} 3 } = \frac x {5 + \frac x 3}$. Then we need only that for $x > 0$
$$\tilde{\phi_4}(x) = (1+x)^3- \left(1+\frac x {5 + \frac x 3} \right)^3 \left(1+ \frac x {3+x} \right)^2 \left(1+ 3 x - 3 \frac x {5 + \frac x 3} - \frac {2 x}{3+x} \right) > 0 \ .$$
Calculation shows that $\tilde{\phi_4}(x) = \frac{x^2}{(1+\frac x 3)^3 (1+\frac x {15})^4 } \sum_{j=0}^8 \gamma_j x^j$, $\gamma_j > 0$ for all $j=0, \cdots , 8$ with $\gamma_0 = \frac {13}{75} > \frac 1 6$. Thus $\tilde{\phi_4}(x) > \frac 1 6 \frac{x^2}{(1+\frac x 3)^3 (1+\frac x {15})^4 } > 0$ for all $x>0$. Hence inequality \eqref{eq3.3} holds for $n=4$, too.
\hfill $\Box$

\vspace{0,8cm}

\section{Proof of Theorem 1.1}

\vspace{0,5cm}

\begin{proof}
The case $\sqrt{\frac{n-1}{2(n+1)}} < t < \sqrt{\frac n {n+1}}$ is covered by Theorem 1.1 of \cite{Ko}. Hence assume that $\sqrt{\frac{n-2}{3(n+1)}} < t \le \sqrt{\frac {n-1}{2(n+1)}}$. We may assume $A(a,t)>0$. As mentioned in the Introduction, the general assumption and $A(a,t)>0$ implies that $a_1 > t$ holds. By Lemma \ref{lem1} $a_3 \le \sqrt{\frac{n-2}{3(n+1)}} < t$, $a_2 \le \sqrt{\frac {n-1}{2(n+1)}}$ and $a_1 \le \sqrt{\frac n {n+1}}$. We have to consider two cases: $a_1 > t \ge a_2 \etc a_{n+1}$ and $a_1 > a_2 > t \ge a_3 \etc a_{n+1}$, with limiting case $a_1=a_2$. In the first case, formula \eqref{eq2.1} contains only one non-zero summand, in the second case two non-zero summands. In the second case $a_2 > \sqrt{\frac{n-2}{3(n+1)}}$, and by Lemma \ref{lem1}
$a_1 \le \psi(a_2) < \sqrt{\frac 2 3}$. \\

i) Assume first that $a_1 > t \ge a_2 \etc a_{n+1}$. By continuity, \eqref{eq2.1} is also valid if some or all of the values $a_2 \etc a_{n+1}$ coincide. To find the maximum of $A(a,t) = \frac{\sqrt{n+1}}{(n-1)!} \frac{(a_1-t)^{n-1}}{\prod_{j=2}^{n+1} (a_1-a_j)}$ in $\Omega := \{ (a_1,a_2 \etc a_{n+1}) \ | \ a_1 > t \ge a_2 \etc a_{n+1} \}$ relative to the constraints $\sum_{j=1}^{n+1} a_j = 0$, $\sum_{j=1}^{n+1} a_j^2 =1$, we first look for the critical points of $A(\cdot,t)$ in the interior of $\Omega$, i.e. when $a_2 \etc a_{n+1} < t$. We show later in ii) that the maximum of $A(\cdot,t)$ is not attained on the boundary of $\Omega$. The Lagrange function
$$L(a,t) = (n-1) \ln(a_1-t) - \sum_{j=2}^{n+1} \ln(a_1-a_j) + \frac {\lambda} 2 \left( \sum_{j=1}^{n+1} a_j^2 - 1 \right) + \mu \left( \sum_{j=1}^{n+1} a_j \right)$$
gives the critical points of $A( \cdot,t)$ relative to the above constraints,
\begin{align*}
\frac{\partial L}{\partial a_1} & = \frac{n-1}{a_1-t} - \sum_{j=2}^{n+1} \frac 1 {a_1-a_j} + \lambda a_1 + \mu = 0 \ , \\
\frac{\partial L}{\partial a_l} & = \frac 1 {a_1-a_l} + \lambda a_l + \mu = 0 \ , \ l = 2 \etc n+1 \ .
\end{align*}
As shown in \cite{Ko}, this implies $\mu = - \frac{n-1}{n+1} \frac 1 {a_1-t}$, $\lambda = n - (n-1) \frac{a_1}{a_1-t}$ and that the $a_l$, $l \ge 2$ are solutions of the quadratic equation $x^2-p x - q =0$, $p = a_1 - \frac{n-1}{n+1} \frac 1 {n t - a_1}$, $q=\frac{(n+1) t - 2 a_1}{(n+1) (n t - a_1)}$. Thus there are at most two different values among the $a_l$'s, $l \ge 2$. In fact, as shown in \cite{Ko}, for $t > \frac 2 {\sqrt{n(n+1)}}$, there is only one possible solution $x_-$ since the other one $x_+$ does not satisfy $x_+ \le t$. In this case, the only critical point of $A(\cdot,t)$ is $a^{(1)} = \sqrt{\frac n {n+1}}(1,-\frac 1 n \etc -\frac 1 n)$. It is a (local) maximum by Proposition 1.3 of \cite{Ko}, and hence the absolute maximum of $A(\cdot,t)$. The condition $t > \frac 2 {\sqrt{n(n+1)}}$ is satisfied for all $n \ge 5$, since then $\sqrt{\frac{n-2}{3(n+1)}} > \frac 2 {\sqrt{n(n+1)}}$. For $n=4$, the condition might be violated, and then there might be two different values among $(a_2,a_3,a_4,a_5)$ in the case of a critical point $a$. Suppose $c:=a_2$ appears $r$ times and $a_5$ $s$ times, $r+s=4$. Then $a_1+rc+sd=0$, $a_1^2+rc^2+sd^2=1$. If $r=s=2$, $c=-\frac 1 4 \left(a_1+\sqrt{4-5a_1^2} \right) < d=-\frac 1 4 \left(a_1-\sqrt{4-5a_1^2} \right)$. Inserting these values into the equation $\frac 1 {a_1-c} + \lambda c + \mu = 0$ with the above values for $\lambda = \lambda(t)$ and $\mu = \mu(t)$ gives a formula for $t$, namely $t= \frac{5a_1^2+2}{20 a_1}$, $a_1 = 2 t + \frac 1 5 \sqrt{100 t^2 -10}$, which is $> \sqrt{\frac 4 5}$ for $t \ge 0.34$. However, $0.365 < \sqrt{\frac 2 {15}} < t$ by our assumption on $t$ for $n=4$. Thus the case $r=s=2$ is impossible. Similarly, also $r=1, s=3$ is impossible. This leaves $r=3, s=1$,
$c=-\frac 1 4 \left(a_1+\frac 1 3 \sqrt{12-15 a_1^2} \right) < d=-\frac 1 4 \left(a_1- \sqrt{12-15 a_1^2} \right)$. In this case $t = \frac{a_1(80 a_1^2+17)+3 \sqrt{12-15 a_1^2}}{40 (8a_1^2-1)} =: \phi(a_1)$. The function $\phi$ is strictly decreasing in $[\sqrt{\frac 2 {15}},\sqrt{\frac 4 5}]$.
One has $\phi(a_1) > \sqrt{\frac 3 {10}}$ for $a_1 < \sqrt{\frac 3 {10}}$, hence $\phi(a_1) = t$ has no solution $a_1 \in [\sqrt{\frac 2 {15}},\sqrt{\frac 3 {10}})$. For $\sqrt{\frac 3 {10}} \le a_1 \le \sqrt{\frac 4 5}$, one has $\phi(\sqrt{\frac 3 {10}}) = \sqrt{\frac 3 {10}}$, and $\phi$ is bijective $\phi : [\sqrt{\frac 3 {10}}, \sqrt{\frac 4 5}] \to [\frac 3 4 \frac 1 {\sqrt 5}, \sqrt{\frac 3 {10}}]$. The equation $t = \phi(a_1) = d$ has the solutions $\bar{a} = \sqrt{\frac 3 {10}} \simeq 0.548$ and $\tilde{a} = \frac{13} 2 \frac 1 {\sqrt{95}} \simeq 0.667$. For $\bar{a} < a_1 < \tilde{a}$ we have $d > t$ which is impossible. This leaves the possibility that $a_1 \in \left[\frac{13} 2 \frac 1 {\sqrt{95}},\sqrt{\frac 4 5} \right]$ with $c < d \le t$. In this case, $0.335 \simeq \frac 3 4 \frac 1 {\sqrt 5} < t < \phi(\tilde{a}) = \frac 4 {\sqrt{95}} \simeq 0.411$. The value $\sqrt{\frac 2 {15}} \simeq 0.365$ is in this range. For $a=(a_1,c,c,c,d) \in S^4$ with $t \in [\sqrt{\frac 2 {15}},\frac 4 {\sqrt{95}}]$ we have \\
$A(a,t)=A(a,\phi(a_1))=\frac{\sqrt 5}{6} \frac{27}{1000} \frac{9 a_1(260 a_1^4-128 a_1^2+17) - 3(20 a_1^4+8 a_1^2-3) \sqrt{12-15 a_1^2}}{(8a_1^2-1)^3} \ . $ \\
Comparing this with \\
$A(a^{(1)},t)=A(a^{(1)},\phi(a_1))=\frac{\sqrt 5}{6} \frac{16}{25} \left(\sqrt{\frac 4 5} - \frac{a_1(80 a_1^2+17)+3\sqrt{12-15a_1^2}}{40(8a_1^2-1)} \right)^3 \ , $ \\
(tedious) estimates show that $A(a,\phi(a_1)) < A(a^{(1)},\phi(a_1))$ for all $a_1 \in \left[\frac{13} 2 \frac 1 {\sqrt{95}}, \sqrt{\frac 4 5} \right)$, with equality for $a_1 = \sqrt{\frac 4 5}$. Thus $A(a,t) < A(a^{(1)},t)$ for all $\sqrt{\frac 2 {15}} \le t < \frac 4 {\sqrt{95}}$. In fact, for $t < \tilde{t} \simeq 0.3877$, $A(a^{(1)},t) < A(a^{(2)},t)$, with $\tilde{t} \in [\sqrt{\frac 2 {15}},\frac 4 {\sqrt{95}}]$.

\vspace{0,5cm}

ii) The boundary of $\Omega$ in i) is $\Lambda = \{ (a_1,a_3 \etc a_{n+1}) \ | \ a_1 > t = a_2 > a_3 \etc a_{n+1} \}$. Note that by Lemma \ref{lem1} i) we have $a_3 \etc a_{n+1} < \sqrt{\frac{n-2}{3(n+1)}} < t$. For $t=a_2$ we have $A(a,t) = \frac{\sqrt{n+1}}{(n-1)!} \frac{(a_1-t)^{n-2}}{\prod_{j=3}^{n+1} (a_1-a_j)}$ with constraints $\sum_{j=1,j \ne 2}^{n+1} a_j = -t$, $\sum_{j=1,j \ne 2}^{n+1} a_j^2 =1 - t^2$. The corresponding Lagrange function yields the necessary conditions for critical points of $A(\cdot,t)$ in $\Lambda$ relative to the constraints
\begin{align}\label{eq4.1a}
\frac{\partial L}{\partial a_1} & = \frac{n-2}{a_1-t} - \sum_{j=3}^{n+1} \frac 1 {a_1-a_j} + \lambda a_1 + \mu = 0 \ , \nonumber \\
\frac{\partial L}{\partial a_l} & = \frac 1 {a_1-a_l} + \lambda a_l + \mu = 0 \ , \ l = 3 \etc n+1 \ ,
\end{align}
implying for $l, k \in \{3 \etc n+1\}$
$$\frac{a_l-a_k}{(a_1-a_l)(a_1-a_k)} = \frac 1 {a_1-a_l} - \frac 1 {a_1-a_k} = \lambda (a_k-a_l) \ . $$
If there are $l \ne k$ with $a_l \ne a_k$, $\lambda = - \frac 1 {(a_1-a_l)(a_1-a_k)}$ and $\mu = - \frac{a_1-a_l-a_k}{(a_1-a_l)(a_1-a_k)}$. Therefore there are at most two different values among the $a_l$, $l \in \{3 \etc n+1\}$. We claim that there is only one value. Suppose to the contrary that $c :=a_l \ne a_k =: d$ satisfy \eqref{eq4.1a} with multiplicities $p$ and $q$, i.e. $p+q=n-1$, $a_1+p c + q d = -t$, $a_1^2 + p c^2 + q d^2 = 1 - t^2$. Inserting the values of $\lambda$ and $\mu$ into the first equation of \eqref{eq4.1a} yields
$$\frac{n-2}{a_1-t} - \frac{p+1}{a_1-c} - \frac{q+1}{a_1-d} = 0 \ , $$
or equivalently
\begin{equation}\label{eq4.1b}
F:= (n-2) (a_1-c) (a_1-d) - (a-t) [ (p+1) (a_1-d) + (q+1) (a_1-d) ] = 0 \ .
\end{equation}
However, we will show that $F>0$ holds. The quadratic equations for $c$ and $d$ imply with $W:= (n-1)(1-a_1^2-t^2)-(a_1+t)^2$, assuming that $c > d$,
$$c=-\frac{a_1+t}{n-1}+\frac 1 {n-1} \sqrt{\frac q p} \sqrt{W} \ , \ d=-\frac{a_1+t}{n-1}-\frac 1 {n-1} \sqrt{\frac p q} \sqrt{W} \ , $$
$$a_1-c = a_1 + \frac{a_1+t}{n-1}-\frac 1 {n-1} \sqrt{\frac q p} \sqrt{W} \ , a_1-d = a_1 + \frac{a_1+t}{n-1} + \frac 1 {n-1} \sqrt{\frac p q} \sqrt{W} \ . $$
Inserting these values into \eqref{eq4.1b} and using $\sqrt{\frac p q} - \sqrt{\frac q p} = \frac{p-q}{\sqrt{p q}}$, $(p+1) \sqrt{\frac p q} - (q+1) \sqrt{\frac q p} = n \frac{p-q}{\sqrt{p q}}$, calculation yields
\begin{align*}
(n-1)^2 F & = \left( (n+1) [ (2n-3) t^2 + a_1 ((n^2-3)t - n a_1) ] - (n-1)(n-2) \right) \\
& + \frac{p-q}{\sqrt{p q}} ( (n^2-2) t - n a_1 ) \sqrt W \ .
\end{align*}
Since $a_1 \le \psi(t) = \sqrt{\frac{n-1} n} \sqrt{1-\frac{n+1} n t^2} - \frac t n < 1.916 t < 2 t$ for all $n \ge 4$, $(n^2-3) t - n a_1 > 0$. Further, $a_1 ((n^2-3)t-na_1)$ is strictly increasing as a function of $a_1 \in [t,\psi(t)]$ for all $n \ge 5$, and for $n=4$ attains its strict minimum at $a_1 = t$. For $a_1 > t \ge \sqrt{\frac{n-2}{3(n+1)}}$, $W < \frac{n+1} 3$. To show $F>0$, the worst case is $p=1$, $q=n-2$, since then $\frac{p-q}{\sqrt{p q}} = - \frac{n-3}{\sqrt{n-2}}$ attains its (negative) minimum. Replacing in this case first $W$ by $\frac{n+1} 3$ and then $a_1$ by $t$, we conclude
$$(n-1)^2 F > (n-2) \left( [(n+3)(n+1) t^2 - (n-1)] - (n-3)(n+1) \sqrt{\frac{n+1}{3(n-2)}} t \right) \ . $$
The right side is increasing in $t$, its derivative being positive for $t \ge \sqrt{\frac{n-2}{3(n+1)}} =: t_0$. Note that $t_0$ depends on $n$, but for simplicity of notation we will omit the index $n$. Replacing $t$ by $\sqrt{\frac{n-2}{3(n+1)}}$, the right side becomes zero and we conclude that $F > 0$. \\
Therefore all $a_l$, $l \ge 3$ satisfying \eqref{eq4.1a} have to coincide, $a_3 = \cdots a_{n+1} = - \frac{a_1+t}{n-1}$, $W=0$, $a_1 = \psi(t)$ and for $a=(a_1,t,a_3 \etc a_3)$
$$\frac{(n-1)!}{\sqrt{n+1}} A(a,t) = \frac{(a_1-t)^{n-2}}{(a_1+\frac{a_1+t}{n-1})^{n-1}} = \left(\frac{n-1} n \right)^{n-1} \frac{(1-\frac t {\psi(t)})^{n-2}}{\psi(t) (1 + \frac 1 n \frac t {\psi(t)})^{n-1}} \ . $$
For $n \ge 4$, $\psi(t) \ge 0.6992$ and $\psi(t) (1 + \frac 1 n \frac t {\psi(t)})^{n-1} \ge (1+ \frac 1 {1.916 n})^{n-1} > 1.01 > 1$, $n \ge 4$. Since $t \ge t_0$, $\psi(t) \le \psi(t_0)$, implying
$$\frac{(n-1)!}{\sqrt{n+1}} A(a,t) < \left(\frac{n-1} n \right)^{n-1} \left( 1 - \frac t {\psi(t_0)} \right)^{n-2} \ . $$
We compare this with
$$\frac{(n-1)!}{\sqrt{n+1}} A(a^{(1)},t) = \left( \frac n {n+1} \right)^{n-\frac 1 2} \left( 1 - \sqrt{\frac{n+1} n} t \right)^{n-1} \ , $$
and want to show that $A(a,t) < A(a^{(1)},t)$ for all $t \in  [\sqrt{\frac{n-2}{3(n+1)}},\sqrt{\frac{n-1}{2(n+1)}}]$ and $n \ge 5$ and for $t \in [0.37,\sqrt{\frac 3 {10}}]$ for $n=4$. The quotient $h_n(t) = \frac{(1-\frac t {\psi(t_0)})^{n-2}}{(1- \sqrt{\frac{n+1} n} t)^{n-1}}$ is decreasing in $[\sqrt{\frac{n-2}{3(n+1)}},\sqrt{\frac{n-1}{2(n+1)}}]$, since
$$(\ln h_n)'(t) = \frac{-(n-2)(\sqrt{\frac{n+1} n} -t) + (n-1)(\psi(t_0) - t)}{(\sqrt{\frac{n+1} n} -t)(\psi(t_0) - t)} < 0 \ . $$
The last inequality holds since for $\phi(t) = -(n-2)(\sqrt{\frac{n+1} n} -t) + (n-1)(\psi(t_0) - t)$ we have $\phi'(t) = -1 < 0$, $\phi(t_0) \le - \frac 1 {20}$, $n \ge 4$. Therefore with $\sqrt{\frac{n+1} n} t_0 = \sqrt{\frac{n-2}{3n}}$
$$\frac{A(a,t)}{A(a^{(1)},t)} \le \frac{A(a,t_0)}{A(a^{(1)},t_0)} < \left( \frac{n^2-1}{n^2}\right)^{n-1} \left(\frac{n+1} n \right)^{\frac 1 2} \frac{(1-\frac{t_0}{\psi(t_0)})^{n-2}}{(1-\sqrt{\frac{n-2}{3 n}})^{n-1}} \ . $$
We have $\frac{\psi(t_0)}{t_0} = \frac{n+1} n \sqrt{2 \frac{n-1}{n-2}} - \frac 1 n < \frac{3n-2}{2n-3}$, the latter inequality being equivalent to
$\sqrt{2 \frac{n-1}{n-2}} < 3 \frac{n-1}{3n-2}$, $2(2n-3)^2 < 9 (n-1)(n-2)$, which is true for all $n \ge 4$. Hence
$$\frac{A(a,t)}{A(a^{(1)},t)} < \left( \frac{n^2-1}{n^2}\right)^{n-1} \left(\frac{n+1} n \right)^{\frac 1 2} \frac{(1- \frac{2n-3}{3n-2})^{n-2}}{(1-\sqrt{\frac{n-2}{3n}})^{n-1}} \ , $$
which is $<1$ for all $n \ge 5$, the right side being decreasing in $n$, as can be seen by taking the logarithmic derivative with respect to $n$: The base in the numerator tends to $\frac 1 3$, the one in the denominator to $1 - \frac 1 {\sqrt 3} > \frac 1 3$ for $n \to \infty$. For $n=4$, $\frac{A(a,t)}{A(a^{(1)},t)} < 1$ for all $t \ge 0.37$. For $t \in [\sqrt{\frac 2 {15}}, 0.37]$ with $\sqrt{\frac 2 {15}} \simeq 0.3651$ we have $A(a,t) < A(a^{(2)},t)$, since the formulas for $A(a,t)$ and $A(a^{(2)},t)$ show for these $t$ and $n=4$
$$A(a,t) \le A(a,\sqrt{\frac 2 {15}}) = \frac 3 {16} \sqrt 2 - \frac 3 {32} \sqrt 6 < 0.0355 < 0.0374 < A(a^{(2)},0.37) \le A(a^{(2)},t) \ . $$
Therefore the maximum of $A(\cdot,t)$ is not attained on the boundary of $\Omega$ in i). \\
Actually, for $\sqrt{\frac 2 {15}} \le t < 0.3668$ we have $A(a^{(1)},t) < A(a,t) < A(a^{(2)},t)$, when $n=4$. \\

\vspace{0,5cm}

iii) Assume next that $a_1 > a_2 > t > \sqrt{\frac{n-2}{3(n+1)}} \ge a_3 \etc a_{n+1}$. Then by Proposition \ref{prop2.1}
$$\frac{(n-1)!}{\sqrt{n+1}} A(a,t) = \frac 1 {a_1-a_2} \left( \frac{(a_1-t)^{n-1}}{\prod_{j=3}^{n+1} (a_1-a_j)} - \frac{(a_2-t)^{n-1}}{\prod_{j=3}^{n+1} (a_2-a_j)} \right) =: \frac{B(a,t)-C(a,t)}{a_1-a_2} \ . $$
By continuity, this formula also holds if some or all of the $a_j$, $j \ge 3$ coincide. Fix $a_1 > a_2 >t$ and consider $(a_1-a_2) A(a,t)$ as a function of $(a_3 \etc a_{n+1})$ with constraints $\sum_{j=3}^{n+1} a_j = -(a_1+a_2)$, $\sum_{j=3}^{n+1} a_j^2 = 1 -(a_1^2+a_2^2)$. The Lagrange function \\
$L(a,\lambda,\mu) := B(a,t) - C(a,t) + \frac \lambda 2 (\sum_{j=3}^{n+1} a_j^2 - 1 +(a_1^2+a_2^2))  + \mu (\sum_{j=3}^{n+1} a_j + (a_1+a_2)) $ \\
yields the critical points of $A$ relative to the constraints via
$$ \frac{\partial L}{\partial a_l} = \frac {B(a,t)} {a_1-a_l} - \frac {C(a,t)} {a_2-a_l} + \lambda a_l + \mu = 0 \ , \ l=3 \etc n+1 \ . $$
This implies that all coordinates $a_l$ satisfy the cubic equation in $x=a_l$
$$B(a,t) (a_2-x) - C(a,t) (a_1-x) + \lambda x (a_1-x)(a_2-x) + \mu (a_1-x)(a_2-x) = 0 \ . $$
Therefore there are at most three different values among the coordinates $a_l$, $l \ge 3$ of a critical point $a \in S^n$. Suppose these occur with multiplicities $p, r, s$, $p+r+s=n-1$, and call them ($x = a_l =$) $\frac c p$, $\frac d r$, $\frac e s$. Then $a_1+a_2+c+d+e=0$, $a_1^2+a_2^2+\frac {c^2} p + \frac {d^2} r + \frac {e^2} s = 1$ and
\begin{align*}
\frac{(n-1)!}{\sqrt{n+1}} A(a,t) &= \frac 1 {a_1-a_2} \left( \frac{(a_1-t)^{n-1}}{(a_1 - \frac c p)^p (a_1 - \frac d r)^r (a_1 - \frac e s)^s} - \frac{(a_2-t)^{n-1}}{(a_2 - \frac c p)^p (a_2 - \frac d r)^r (a_2 - \frac e s)^s} \right) \\
&=: \frac{B - C}{a_1-a_2} \ .
\end{align*}
Suppose that $\frac c p < \frac d r < \frac e s$. Then $c<0$ and by Lemma \ref{lem1} $a_1 \le \sqrt{\frac n {n+1}}$, $a_2 \le \sqrt{\frac{n-1}{2(n+1)}}$. We perturb the values $c, d, e$ by $\delta > \varepsilon := \theta \delta > 0$, $\theta := \frac{\frac d r - \frac c p}{\frac e s-\frac c p} \in (0,1)$ in the following way: $c' = c + (\delta - \varepsilon)$, $d' = d -\delta$, $e' = e+\varepsilon$. Then $a_1+a_2+c'+d'+e' =0$, $a_1^2+a_2^2+\frac{c'^2} p + \frac{d'^2} r + \frac{e'^2} s = 1 + O(\delta^2)$, since $\frac c p (\delta - \varepsilon) - \frac d r \delta + \frac e s \varepsilon = 0$.
These values are independent of $a_1$. Calculation shows that
$$g_{a_1}(\delta) := \frac 1 {(a_1 - \frac c p - (1-\theta) \frac \delta p)^p} \frac 1 {(a_1 - \frac d r +  \frac \delta r)^r} \frac 1 {(a_1 - \frac e s -  \theta \frac \delta s)^s}$$
satisfies
$$g_{a_1}(\delta) = g_{a_1}(0) + \frac{(\frac e s- \frac d r)(\frac d r- \frac c p)}{(a_1 - \frac c p)^{p+1}(a_1 - \frac d r)^{r+1}(a_1 - \frac e s)^{s+1}} \delta + O(\delta^2) \ , $$
where the terms in the numerator are positive. (The value of $\varepsilon$ could be modified $\varepsilon = \theta \delta + O(\delta^2)$ to guarantee that $a_1^2+a_2^2+\frac{c'^2} p + \frac{d'^2} r + \frac{e'^2} s = 1$ holds precisely. The additional second order perturbation yields the same result for $g_{a_1}(\delta)$.)  Let $B(\delta) := (a_1-t)^{n-1} g_{a_1}(\delta)$ and $C(\delta) := (a_2-t)^{n-1} g_{a_2}(\delta)$. We claim that the first order perturbation of $B$, $B(\delta)-B(0)$, is bigger than the one of $C$, $C(\delta)- C(0)$, so that $B(\delta) - C(\delta) > B(0) - C(0)$, as $\delta \to 0$ (up to $O(\delta^2)$). This is equivalent to the inequality
$(a_1-t)^{n-1} g_{a_1}(\delta) > (a_2-t)^{n-1} g_{a_2}(\delta)$ for small $\delta$, i.e.
$$\left(\frac{a_1-t}{a_2-t} \right)^{n-1} > \left(\frac{a_1- \frac c p}{a_2- \frac c p} \right)^{p+1} \left(\frac{a_1- \frac d r}{a_2- \frac d r} \right)^{r+1} \left(\frac{a_1- \frac e s}{a_2- \frac e s} \right)^{s+1} \ , $$
$p+r+s=n-1$. By Proposition \ref{prop3.1} this inequality is true. Therefore, fixing $a_1 > a_2 >t$, $A(a,t)$ increases if $c$ is replaced by $c+\delta_1$, $d$ by $d - \delta_2$ and $e$ by $e+\delta_3$ for small $\delta_1, \delta_2, \delta_3 >0$ depending on each other. This holds independently of $t$. Therefore the maximum of $A(a,t)$ relative to fixed $a_1>a_2>t$ occurs for $\frac c p = \frac d r < \frac e s$. Further, we claim that in the maximum case, the multiplicity of $\frac e s$ should be $s=1$. If $s \ge 2$, the vector $a$ would have $p$ coordinates $\frac c p$ and $s$ coordinates $\frac e s$ with $\frac c p < \frac e s$, $p+s=n-1$, $c+e = -(a_1+a_2)$, $\frac{c^2} p + \frac {e^2} s = 1 - a_1^2 -a_2^2$. Define a perturbed vector $a'$ having $p$ coordinates $\frac{c'} p = \frac{c+\varepsilon(\delta)} p$, $s-1$ coordinates $\frac{d'}{s-1} = \frac e s - \frac{\delta}{s-1}$ and $1$ coordinate $e' = \frac e s + \delta-\varepsilon(\delta)$, where $\varepsilon(\delta) = \frac 1 2 \frac s {s-1} \frac{\delta^2}{\frac e s - \frac c p}$ for small $\delta > 0$, $0 < \varepsilon(\delta) < \delta$. Then $c'+d'+e' = c + e = -(a_1+a_2)$ and $\frac{c'^2} p + \frac{d'^2}{s-1} + e'^2 = \frac{c^2} p + \frac{e^2} s + O(\delta^3) = 1 -a_1^2-a_2^2 + O(\delta^3)$. Let
$$g_{a_1}(\delta) := \frac 1 {(a_1 - \frac c p -  \frac {\varepsilon(\delta)} p)^p} \frac 1 {(a_1 - \frac e s +  \frac \delta {s-1})^{s-1}} \frac 1 {(a_1 - \frac e s - \delta + \varepsilon(\delta))^s} \ . $$
Calculation shows that
$$g_{a_1}(\delta) - g_{a_1}(0) = \frac 1 2 \frac{\frac e s - \frac c p}{(a_1-\frac c p)^{p+1}(a_1-\frac e s)^{s+2}} + O(\delta^3) \ , $$
with $\frac e s - \frac c p > 0$ being independent of $a_1$. We remark that if $\varepsilon(\delta)$ is chosen as the solution of the quadratic equation implied by requiring that $a'$ satisfies the constraints exactly, and not only up to $O(\delta^3)$, the same is true. The given value $\varepsilon(\delta)$ is an $O(\delta^2)$-approximation of this solution. Again we have $B(\delta)-C(\delta) > B(0)-C(0)$, i.e. $A(a',t) > A(a,t)$, since
$$\left(\frac{a_1-t}{a_2-t} \right)^{n-1} > \left(\frac{a_1-\frac c p}{a_2-\frac c p} \right)^{p+1} \left(\frac{a_1-\frac e s}{a_2- \frac e s} \right)^{s+2} \ . $$
The last inequality follows from Proposition \ref{prop3.1} with the choice $r=s-1$, $s=1$, $c+(s-1)\frac e s + \frac e s = - (a_1+a_2)$ and
$p (\frac c p)^2 + (s-1) (\frac e s)^2 + (\frac e s)^2 = \frac {c^2} p + \frac{e^2} s = 1-a_1^2-a_2^2$. The vector $a'$ again has three different coordinates, the largest one of multiplicity $1$. Applying the first case then shows that in the maximum case the two smaller coordinates should coincide. \\
For $e$ of multiplicity $1$ and $\frac c p$ of multiplicity $n-2$ we have
$$\frac{(n-1)!}{\sqrt{n+1}} A(a,t) = \frac 1 {a_1-a_2} \left( \frac{(a_1-t)^{n-1}}{(a_1 - \frac c {n-2})^{n-2} (a_1 - e)} - \frac{(a_2-t)^{n-1}}{(a_2 - \frac c {n-2})^{n-2} (a_2 - e)} \right) \ , $$
with $a_1+a_2+c+e = 0$, $a_1^2+a_2^2+ \frac {c^2}{n-2} + e^2 = 1$. This yields $c = -\frac{n-2}{n-1} (a_1+a_2) - \frac{\sqrt{n-2}}{n-1} W$, $e= -\frac 1 {n-1} (a_1+a_2) + \frac{\sqrt{n-2}}{n-1} W$, where $W = \sqrt{(n-1)(1-a_1^2-a_2^2)-(a_1+a_2)^2}$, with $\frac c {n-2} < e$, $c<0$. The square root $W$ is well-defined, since by Lemma \ref{lem1} $a_2 \le \psi(a_1) := \sqrt{\frac{n-1} n} \sqrt{1 - \frac{n+1} n a_1^2} - \frac {a_1} n$ or equivalently $a_1 \le \psi(a_2)$. In fact, $\psi$ is strictly decreasing and $\psi : [\sqrt{\frac{n-2}{3(n+1)}},\psi(\sqrt{\frac{n-2}{3(n+1)}})] \to [\sqrt{\frac{n-2}{3(n+1)}},\psi(\sqrt{\frac{n-2}{3(n+1)}})]$ is bijective with $\psi^{-1} = \psi$ and $\psi(\sqrt{\frac{n-2}{3(n+1)}}) = \sqrt{\frac 2 3} \frac{\sqrt{n^2-1}} n - \frac 1 n \sqrt{\frac{n-2}{3(n+1)}} < \sqrt{\frac 2 3}$. Further, $\psi(\sqrt{\frac{n-1}{2(n+1)}}) = \sqrt{\frac{n-1}{2(n+1)}}$. \\
Inserting the values of $c$ and $e$, we have that
\begin{align}\label{eq4.1}
\frac{(n-1)!}{\sqrt{n+1}} A(a,t) & = \frac 1 {a_1-a_2} \Big( \frac{(a_1-t)^{n-1}}{(a_1 + \frac{a_1+a_2}{n-1} + \frac 1 {\sqrt{n-2} (n-1)} W)^{n-2} (a_1 + \frac{a_1+a_2}{n-1} - \frac{\sqrt{n-2}}{n-1} W) } \\
& \quad \quad \quad \quad - \frac{(a_2-t)^{n-1}}{(a_2 + \frac{a_1+a_2}{n-1} + \frac 1 {\sqrt{n-2} (n-1)} W)^{n-2} (a_2 + \frac{a_1+a_2}{n-1} - \frac{\sqrt{n-2}}{n-1} W) } \Big)  \nonumber \\
& =: \frac 1 {a_1-a_2} \left(B(a_1,a_2)-C(a_1,a_2) \right) \ . \nonumber
\end{align}
By Proposition \ref{prop3.2} $\frac {B(a_1,a_2)-C(a_1,a_2)}{a_1-a_2}$ is increasing in $a_2$. Thus the maximum of $A(a,t)$ occurs either for $a_2 \to a_1$, if $a_1 \le \sqrt{\frac{n-1}{2(n+1)}}$ or for $a_2 = \psi(a_1)$, if $a_1 > \sqrt{\frac{n-1}{2(n+1)}}$. Recall here that $\sqrt{\frac{n-2}{3(n+1)}} < a_2 \le \sqrt{\frac{n-1}{2(n+1)}}$, so that in the latter situation $a_2 = a_1$ is impossible. \\

\vspace{0,5cm}

iv) We first consider the subcase of iii) that $t < a_1 \le \sqrt{\frac{n-1}{2(n+1)}}$ and $\sqrt{\frac{n-2}{3(n+1)}} < a_2$. Then the limit $a_2 \to a_1$ can be taken in \eqref{eq4.1} using l'Hospital's rule, yielding with $V(n,a_1) := \sqrt{n-1-2(n+1) a_1^2}$ and $\tilde{a} = (a_1,a_1,c(a_1), \cdots , c(a_1),d(a_1))$,
$c(a_1) = - \frac{2a_1}{n-1} - \frac{V(n,a_1)}{\sqrt{n-2} (n+1)}$ and $d(a_1) = - \frac{2a_1}{n-1} + \frac{\sqrt{n-2} V(n,a_1)}{n+1}$ that
\begin{align}\label{eq4.2a}
& (\frac{n+1}{n-1})^n \frac{(n-1)!}{\sqrt{n+1}} A(\tilde{a},t) = \\
& = \frac{(a_1-t)^{n-2}} {\left(a_1+\frac 1 {\sqrt{n-2} (n+1)} V(n,a_1) \right)^{n-1} \left(a_1 - \frac{\sqrt{n-2}}{n+1} V(n,a_1) \right)^2}  M(n,t,a_1) = : f(n,t,a_1) \ , \nonumber
\end{align}
where $ M(n,a_1) := 4 a_1^2 + (n-1) a_1 t - \frac{n-1}{n+1} - \frac{n-3}{\sqrt{n-2} (n+1)} [2 a_1 + (n-1) t] V(n,a_1)$.
We claim that $f(n,t,a_1)$ is strictly increasing as a function of $a_1$. Let
\begin{align*}
f_1(n,t,a_1) & := \frac{(a_1-t)^{n-3}}{\left(a_1+\frac 1 {\sqrt{n-2} (n+1)} V(n,a_1) \right)^{n-2} \left(a_1 - \frac{\sqrt{n-2}}{n+1} V(n,a_1) \right)} \\
f_2(n,t,a_1) & := \frac{(a_1-t)}{\left(a_1+\frac 1 {\sqrt{n-2} (n+1)} V(n,a_1) \right) \left(a_1 - \frac{\sqrt{n-2}}{n+1} V(n,a_1) \right)} M(n,t,a_1) \ .
\end{align*}
Then $f(n,t,a_1) = f_1(n,t,a_1) f_2(n,t,a_2)$. We will show that for $n \ge 5$ the functions $f_1$ and $f_2$ are both increasing in $a_1$. For $n=4$ we have to multiply $f_1$ by $\sqrt{a_1-t}$ and divide $f_2$ by $\sqrt{a_1-t}$ to get the same result. Since $f_1, f_2 > 0$, it suffices to prove that
$(\ln f_i)'(n,t,a_1) >0$ for $a_1 > t$ and $i=1,2$, where $'$ denotes the derivative with respect to $a_1$. As for $f_1$ we have, using $V'(n,a_1) = - \frac{2 (n+1) a_1}{V(n,a_1)}$,
\begin{align*}
(\ln f_1)'(n,t,a_1) & = \frac{n-3}{a_1-t} - \frac{n-2 - \frac{2a_1 \sqrt{n-2}}{V(n,a_1)}}{a_1+\frac{V(n,a_1)}{\sqrt{n-2} (n+1)}} - \frac{1 + \frac{2a_1 \sqrt{n-2}}{V(n,a_1)}}{a_1-\frac{\sqrt{n-2} V(n,a_1)}{n+1}} \\
& = \frac{n-3}{a_1-t} - \frac{(n-1) A}{a_1 A - \frac{n-1}{n+1}} \ , \ A:= (n+3)a_1 - \frac{n-3}{\sqrt{n-2}} V(n,a_1) \ .
\end{align*}
This is positive if and only if
$$1 < \frac{n+1}{n-3} \left(t - \frac{2 a_1}{n-1} \right) A = \frac{n+1}{n-3} \left(t - \frac{2 a_1}{n-1} \right) \left((n+3)a_1 - \frac{n-3}{\sqrt{n-2}} V(n,a_1) \right) \ .$$
Assume first that $n \ge 5$ holds. For the last inequality it suffices to verify
$$1 <  \frac{n+1}{n-3} \left(t - \frac{2 a_1}{n-1} \right) \left((n+3)a_1 - \frac{n-3}{\sqrt{n-2}} V(n,a_1) \right) =:g_n(a_1) \ $$
for all $a_1 \in (\sqrt{\frac{n-2}{3(n+1)}},\sqrt{\frac{n-1}{2(n+1)}}] =: I_n$, since in this subcase $a_1 \le \sqrt{\frac{n-1}{2(n+1)}}$ is assumed. We have $g_n(\sqrt{\frac{n-2}{3(n+1)}}) = 1$ and we will show for all $n \ge 5$ that $g_n'(a_1) > 0$ for all $a_1 \in I_n$.  We find
\begin{align}\label{eq4.2b}
\frac{g_n'(a_1)}{n+3} & = \left(\sqrt{\frac{n-2}{3(n+1)}} - \frac {4 a_1}{n-1} \right)  \nonumber \\
& + \frac{2(n-3)}{(n+3) \sqrt{n-2}} \left[ \frac{(n+1) a_1}{V(n,a_1)} \left(\sqrt{\frac{n-2}{3(n+1)}} - \frac {2 a_1}{n-1} \right) + \frac{V(n,a_1)}{n-1} \right] \ .
\end{align}
Since $\sqrt{\frac{n-2}{3(n+1)}} - \frac {4 a_1}{n-1}  \ge \sqrt{\frac{n-2}{3(n+1)}} - 2 \sqrt{\frac 2 {n^2-1}} =: k(n)  > 0$ if and only if $(n-1)(n-2) > 24$, i.e. $n \ge 7$, $g_n'(a_1) > 0 $ holds for all $n \ge 7$, since the bracket $[ \cdot ]$ is always positive. For $n=5, 6$, $k(5) = \frac 1 {\sqrt 6} - \frac 1 {\sqrt 3} \simeq -0.169$, $k(6) = 2 (\frac 1 {\sqrt{21}} - \sqrt{\frac 2 {35}} ) \simeq - 0.0417$. For $n=5,6$ we have to estimate $[ \cdot ]$ from below. Easy estimates show that \newline $a_1 (\sqrt{\frac{n-2}{3(n+1)}} - \frac {2 a_1}{n-1} ) \ge \alpha_n$, where $\alpha_5 = \frac{\sqrt 2 - 1} 6$, $\alpha_6 = \frac 4 {35}$ for all $a_1 \in I_n$. Therefore the bracket $[ \cdot ]$ in \eqref{eq4.2b} is $\ge \phi(V(n,a_1))$, where $\phi(s) := \frac A s + B s$, $s>0$, with $A = (n+1) \alpha_n$ and $B = \frac 1 {n-1}$. The minimum of $\phi$ in $(0,\infty)$ is $2 \sqrt{A B} = 2 \sqrt{\frac{n+1}{n-1} \alpha_n}$, attained at $s = \sqrt{\frac A B}$. Hence for $n=5, 6$, $\frac{g_n'(a_1)}{n+3} \ge \frac{2(n-3)}{(n+3) \sqrt{n-2}} 2 \sqrt{\frac{n+1}{n-1} \alpha_n}$. Inserting the values for $n, \alpha_n$ yields that $g_5'(a_1) > \frac 2 {15} > 0$ and $g_6'(a_1) > 2 > 0$. Thus $g_n(a_1) > 1$ for $a_1 \in I_n$ and hence $f_1$ is strictly increasing for $n \ge 5$. Similarly for $n=4$, $\sqrt{a_1 - t} \ f_1(n,t,a_1)$ is increasing in $a_1$. Concerning $f_2$ we have
\begin{align*}
N(n,t,a_1) & := M'(n,t,a_1) = 8 a_1 + (n-1) t  \\
& + 2 \frac{n-3}{\sqrt{n-2} \ V(n,a_1)} \left(4 a_1^2+ (n-1) a_1 t - \frac{n-1}{n+1} \right) \ , \\
M_2(n,t,a_1) & := (n+1) \left(a_1 + \frac {V(n,a_1)} {\sqrt{n-2} \ (n+1)} \right) \left(a_1 - \frac{\sqrt{n-2} \ V(n,a_1)} {n+1} \right) \\
& = (n+3) a_1^2 - \frac{n-3}{\sqrt{n-2}} a_1 V(n,a_1) - \frac{n-1}{n+1} \ .
\end{align*}
We get
\begin{align*}
(\ln f_2)'& (n,t,a_1) = \frac 1 {a_1-t} - \frac{2(n+3)a_1-\frac{n-3}{\sqrt{n-2}}V(n,a_1) + 2 \frac{(n-3)(n+1)}{\sqrt{n-2}} \frac{a_1^2}{V(n,a_1)} } {M_2(n,t,a_1)} + \frac{N(n,t,a_1)}{M(n,t,a_1)}  \\
& = \left(\frac 1 {a_1 - t} - \frac{2(n+3)a_1-\frac{n-3}{\sqrt{n-2}}V(n,a_1)}{M_2(n,t,a_1)} \right) + \left(\frac{N(n,t,a_1)}{M(n,t,a_1)} - \frac{N_2(n,t,a_1)}{M_2(n,t,a_1)} \right) \ ,
\end{align*}
with $N_2(n,t,a_1) := 2 \frac{(n-3)(n+1)}{\sqrt{n-2}} \frac{a_1^2}{V(n,a_1)}$. We claim that both terms in the brackets are positive. Calculation shows that the first term in brackets is positive if and only if $ t > \frac{(n+3) a_1^2 + \frac{n-1}{n+1}}{2(n+3)a_1 - \frac{n-3}{\sqrt{n-2}} V(n,a_1)} =: H_n(a_1) $. We have $H_n(\sqrt{\frac{n-2}{3(n+1)}}) = \sqrt{\frac{n-2}{3(n+1)}}$ and claim that for all $n \ge 5$, the function $H_n$ is strictly decreasing in $I_n$. Then for $a_1 > t$, $H_n(a_1) < \sqrt{\frac{n-2}{3(n+1)}} \le t$ will be satisfied. As for $H_n'$, we find with $\frac{n-1}{n+1} \ge 2 a_1^2$
\begin{align*}
(\ln H_n)'(a_1) & = \frac{2(n+3) a_1}{(n+3) a_1^2 + \frac{n-1}{n+1}} - \frac{1 + \frac{n-3}{(n+3) \sqrt{n-2}} \frac{(n+1) a_1}{V(n,a_1)}}{2(n+3)a_1 - \frac{n-3}{\sqrt{n-2}} V(n,a_1)}  \\
& \le \frac 1 {(n+5) a_1} - \frac{1 + \frac{n-3}{(n+3) \sqrt{n-2}} \frac{(n+1) a_1}{V(n,a_1)}}{2(n+3)a_1 - \frac{n-3}{\sqrt{n-2}} V(n,a_1)} \ .
\end{align*}
Calculation shows that the latter is negative if and only if \\
$1 \le \frac{n-3}{\sqrt{n-2}} (\frac{V(n,a_1)}{(n+1) a_1} + \frac{n+5}{n+3} \frac {a_1} {V(n,a_1)}) =  \frac{n-3}{\sqrt{n-2}} \phi(\frac{a_1}{V(n,a_1)}) =: l_n(a_1)$ with $A:= \frac 1 {n+1}$ and $B := \frac{n+5}{n+3}$ in terms of the above function $\phi$. Therefore $l_n(a_1) \ge \frac{n-3}{\sqrt{n-2}} 2 \sqrt{A B} = \frac{n-3}{\sqrt{n-2}} 2 \sqrt{\frac{n+5}{(n+1)(n+3)}} =: \Phi(n)$. We have that $\Phi$ is increasing with $\Phi(5) = \frac{\sqrt{10}} 3 > 1$. Therefore $H_n$ is decreasing for all $n \ge 5$. The second term in brackets is positive if and only if
$g(n,t,a_1) := N(n,t,a_1) M_2(n,t,a_1) - N_2(n,t,a_1) M(n,t,a_1) > 0$. Calculation shows
\begin{align*}
g(n,t,a_1) & = \left((n+1)(n+3)a_1^2-\frac{(n-1)^2}{n+1} \right) t + \\
& \frac{2 a_1}{n-2} \left(2(n^2+8n-21)a_1^2 + \frac{n-1}{n+1}(n^2-10n+17) \right) \\
 & - \frac{n-3}{\sqrt{n-2} (n+1)} V(n,a_1) \left[8(n+2)a_1^2 + (n-1)(n+3) a_1 t - 2 \frac{n-1}{n+1} \right] \ .
\end{align*}
Since $V(n,a_1)$ is decreasing in $a_1$, $g$ is obviously strictly increasing in $a_1$ and in $t$. We have for $t = a_1 = \sqrt{\frac{n-2}{3(n+1)}}$ that $g(n,t,a_1)=0$. Thus for $t < a_1$, $g(n,t,a_1)>0$, and hence $f_2$ is increasing in $a_1$, too.  For $n=4$, $\frac{f_2(n,t,a_1)}{\sqrt{a_1 - t}}$ is increasing in $a_1$ and hence $f(4,t,a_1)$ is increasing in $a_1$, too. \\

We conclude that for $\sqrt{\frac{n-2}{3(n+1)}} \le t < a_1 < \sqrt{\frac{n-1}{2(n+1)}}$, $A(\tilde{a},t)$ is increasing in $a_1$. Since for $a_1 \to \sqrt{\frac{n-1}{2(n+1)}}$, $V(n,a_1) \to 0$ and hence $\tilde{a} \to a^{(2)}$, we find from \eqref{eq4.2a} \\
$A(\tilde{a},t) \le  A(a^{(2)},t) = \frac{\sqrt{n+1}}{(n-2)!} \left(\sqrt{\frac {2(n-1)} {n+1}} \right)^n \left(\sqrt{\frac 2 {(n-1)(n+1)}}+t \right) \left(\sqrt{\frac {n-1} {2(n+1)}} - t \right)^{n-2}$, \\ if $t < \sqrt{\frac{n-1}{2(n+1)}}$, coinciding, of course, with the value for $A(a^{(2)},t)$ given in Proposition \ref{prop2.2}. Therefore in this case $a^{(2)}$ yields the maximum of $A(\cdot,t)$.

\vspace{0,5cm}

v) Secondly, consider the subcase of iii) that $a_1 > \sqrt{\frac{n-1}{2(n+1)}} =:t_1$. By Lemma \ref{lem1}, $a_2 \le \psi(a_1) \le \psi(t_1) = t_1$. Thus $a_1 > t_1 \ge a_2 > t > t_0 := \sqrt{\frac{n-2}{3(n+1)}}$. Since $A(\cdot,t)$ is increasing in $a_2$, the maximum occurs for $a_2 = \psi(a_1)$. Then $W = 0$ and
\begin{equation}\label{eq4.2}
\frac{(n-1)!}{\sqrt{n+1}} A(a,t) = \frac{ \left(\frac{a_1-t} {a_1+ \frac{a_1+a_2} {n-1} } \right)^{n-1} - \left(\frac{a_2-t} {a_2+ \frac{a_1+a_2} {n-1} } \right)^{n-1} }{a_1 - a_2} =: B(a_1,a_2,t) \ .
\end{equation}
For $a_1 > t_1$, $\psi'(a_1) = -\frac{n+1} n \frac{\sqrt{n-1} a_1}{\sqrt{n-(n+1) a_1^2}} - \frac 1 n \le -1$. Since $\psi(t_1) = t_1$, we have for $a_1 > t_1$ and a suitable $\theta \in (t_1,a_1)$
$$a_1 - a_2 = a_1 - \psi(a_1) = a_1 - t_1 - (\psi(a_1) - \psi(t_1)) = (a_1 - t_1) (1 - \psi'(\theta)) \ge 2 (a_1 - t_1) \ . $$
We will show for $n \ge 5$ that $A(a,t) < A(a^{(1)},t)$, considering two cases: first, $a_1 \ge \sqrt{\frac{n-\frac 1 2}{2(n+1)}} =: t_2$ and secondly $t_1 < a_1 < t_2$. In the first case, $a_1$ and $a_2 = \psi(a_1)$ are sufficiently far apart that we can omit the second term in $B(a_1,a_2,t)$. In the second case, we use the mean value theorem to estimate $B(a_1,a_2,t)$. The second case is investigated in part vi) below, the case $n=4$ in part vii). \\

Thus let $n \ge 5$ and $a_1 \ge t_2 > t_1 > a_2 = \psi(a_1) > t > t_0$. By \eqref{eq4.2}
\begin{align}\label{eq4.3}
\frac{(n-1)!}{\sqrt{n+1}} A(a,t) & \le \frac 1 {2(a_1-t_1)} \left(\frac{a_1-t} {a_1+ \frac{a_1+\psi(a_1)} {n-1} } \right)^{n-1}
= \frac {(\frac{n-1} n)^{n-1}} {2(a_1-t_1)}  \left(\frac{1- \frac t {a_1}}{1+\frac 1 n \frac{\psi(a_1)}{a_1}} \right)^{n-1} \\
& \le \frac {(\frac{n-1} n)^{n-1}} {2(a_1-t_1)} \left(\frac{1- \frac t {a_1}}{1+\frac 1 n \frac t {a_1}} \right)^{n-1} =: L \ , \nonumber
\end{align}
Since $\psi^2 = \Id$ and $\psi$ is decreasing, this implies
$a_1 = \psi(a_2) \le \psi(t_0) = \sqrt{\frac 2 3} \frac{\sqrt{n^2-1}} n - \frac 1 n t_0 < \sqrt{\frac 2 3} \sqrt{\frac n {n+1}} =: t_3$. The last inequality follows from $\frac{n^2-1}{n^2} < (\sqrt{1-\frac 1 {n+1}} + \frac 1 n t_0)^2 = 1 - \frac 1 {n+1} + \frac{t_0^2}{n^2} + \frac 2 {\sqrt 3} \frac 1 {n+1} \sqrt{\frac{n-2} n}$, which is implied by $1 - \frac 1 n < \frac 2 {\sqrt 3} \sqrt{\frac{n-2} n}$ for $n \ge 5$. We want to show that $L$ is strictly less than
$$\frac{(n-1)!}{\sqrt{n+1}} A(a^{(1)},t) = \left(\frac n {n+1} \right)^{n-\frac 1 2} \left(1 - \sqrt{\frac{n+1} n} t \right)^{n-1} =: R \ , $$
i.e. that \eqref{eq4.2} does not yield the maximum of $A(\cdot,t)$. We have
$$ \frac L R = \frac {(\frac{n^2-1} {n^2})^{n- \frac 1 2} \sqrt{\frac n {n-1}} } {2(a_1-t_1)} \left(\frac{1- \frac t {a_1}}{(1-\sqrt{\frac{n+1} n} t)(1+\frac 1 n \frac t {a_1})} \right)^{n-1} \ . $$
Now $h(t) := \ln(\frac{1- \frac t {a_1}} {(1-\sqrt{\frac{n+1} n} t)(1+\frac 1 n \frac t {a_1})})$ is decreasing in $t$, since \\
$h'(t) = -\frac 1 {a_1-t}+\frac 1 {\sqrt{\frac n {n+1}} - t} - \frac {\frac 1 n} {a_1 + \frac 1 n t} = - \frac{\sqrt{\frac n {n+1}} - a_1}{(a_1-t)(\sqrt{\frac n {n+1}} - t)} - \frac {\frac 1 n} {a_1 + \frac 1 n t} < 0$,
so that $h$ is maximal for $ t = t_0 := \sqrt{\frac{n-2}{3(n+1)}}$. Therefore
$$\frac L R \le \frac  {(\frac{n^2-1} {n^2})^{n- \frac 1 2} \sqrt{\frac n {n-1}} }{2(a_1-t_1)} \left(\frac{1- \frac{t_0}{a_1}} {(1-\sqrt{\frac{n+1} n} t_0) (1+\frac 1 n \frac{t_0}{a_1})} \right)^{n-1} =: Q_n(a_1) . $$
We claim that $\phi(a_1) := \frac 1 {a_1-t_1} \left(\frac{1- \frac{t_0}{a_1}} {(1-\sqrt{\frac{n+1} n} t_0) (1+\frac 1 n \frac{t_0}{a_1})} \right)^{n-1}$ is first decreasing and then increasing in $(t_1,t_3]$. This will imply $Q_n(a_1) \le \max( Q_n(t_2) , Q_n(t_3) )$ for all $a_1 \in [t_2,t_3]$. The logarithmic derivative of $\phi$ is
\begin{align*}
(\ln \phi)'(a_1) &= - \frac 1 {a_1-t_1} + \frac{(n-1) t_0}{a_1 (a_1-t_0)} + \frac{(n-1) t_0}{a_1 (na_1+t_0)} \\
& = \frac{n t_0 a_1 -a_1^2 -(n-1) t_0 t_1}{(a_1(a_1-t_0)(a_1-t_1)} + \frac{(n-1) t_0}{a_1 (na_1+t_0)} \ .
\end{align*}
The first term is negative for $a_1 > t_1$ close to $t_1$, being singular at $t_1$. The quadratic term
$$q(a_1) = n t_0 a_1 -a_1^2 -(n-1) t_0 t_1 = \sqrt{\frac{n-2}{n+1}} \left( \frac{n a_1} {\sqrt 3} - \frac{(n-1)^{\frac 3 2}}{\sqrt 6 (n+1)^{\frac 3 2}} \right) -a_1^2$$
is positive at $a_1 = t_3 = \sqrt{\frac 2 3 \frac n {n+1}}$. The second zero of $q$ is near $n t_0 - t_1$, clearly $>1$ for $n \ge 5$. Thus the first term in
$(\ln \phi)'(a_1)$ has exactly one zero $\tilde{a} \in (t_1,t_3)$. It is decreasing in modulus for $a_1 \in (t_1,\tilde{a})$, since $q$ is increasing there, with negative values, and the denominator $a_1 (a_1-t_0) (a_1-t_1)$ is increasing, too. Therefore also $(\ln \phi)'(a_1)$ has exactly one zero $\bar{a} \in (t_1,t_3)$ with $\bar{a} < \tilde{a}$, since we add a second positive term. This means that $\phi$ is first decreasing and then increasing in $(t_1,t_3)$. Therefore
$$\frac L R \le Q_n(a_1) \le \max ( Q_n(t_2) , Q_n(t_3) ) \ , \ a_1 \in [t_2,t_3] \ . $$
To estimate $Q_n(t_3)$, we use
$\frac 1 {2(t_3-t_1)} = \frac{\sqrt{n+1}}{n+3} (\sqrt 6 \sqrt n + \frac 3 {\sqrt 2} \sqrt{n-1} ) \le (\sqrt 6 + \frac 3 {\sqrt 2}) \frac{n+\frac 1 2}{n+3} < 4.6 - \frac{11}{n+3} \ . $
Logarithmic differentiation shows that $f_1(n) :=(1+\frac 1 {\sqrt 2 n} \sqrt{\frac{n-2} n})^{n-1}$ is increasing in $n$, so that $f_1(n) \ge f_1(5) > \frac 3 2$ for $n \ge 5$. We also use $f_2(n) := (\frac{n^2-1} {n^2})^{n- \frac 1 2} \sqrt{\frac n {n-1}} < 1$. Note that $\frac{t_0}{t_3} = \frac 1 {\sqrt 2} \sqrt{\frac{n-2} n}$ and $\sqrt{\frac{n+1} n} t_0 = \frac 1 {\sqrt 3} \sqrt{\frac{n-2} n}$. The function $f_3(\alpha) := \frac{1-\frac{\alpha}{\sqrt 2}}{1-\frac{\alpha}{\sqrt 3}}$ is decreasing in $\alpha \in (0,1)$, since $(\ln f_2)'(\alpha) = - \frac{\sqrt 3 - \sqrt 2}{(\sqrt 3 - \alpha)(\sqrt 2 - \alpha)} < 0$. Applying this for $\alpha = \sqrt{\frac{n-2} n}$, $f_4(n) := \frac{1- \sqrt{\frac{n-2} {2n}}}{1- \sqrt{\frac{n-2} {3n}}}$ satisfies $f_4(n) \le f_4(5) < 0.82$ for all $n \ge 5$. We conclude that
$$Q_n(t_3) < (4.6 - \frac{11}{n+3}) \frac{f_4(n)^{n-1}}{f_1(n)} < \frac 2 3 (4.6 - \frac{11}{n+3}) 0.82^{n-1} =: f_5(n) \ , $$
which is decreasing in $n$, with $f_5(n) \le f_5(5) < 0.98 < 1$ for all $n \ge 5$. \\

As for $Q_n(t_2)$, we have $\frac{t_0}{t_2} = \sqrt{\frac 2 3 \frac{n-2}{n-\frac 1 2}}$, $\sqrt{\frac{n+1} n} t_0 = \sqrt{\frac{n-2} {3n}}$. By the arithmetic-geometric mean inequality
$\frac 1 {2(t_2-t_1)} = \sqrt 2 \sqrt{n+1} (\sqrt{n-\frac 1 2}+\sqrt{n-1}) \le 2 \sqrt 2 n + \frac 1 {2 \sqrt 2}$. Therefore
$$Q_n(t_2) \le (2 \sqrt 2 n + \frac 1 {2 \sqrt 2}) \left(\frac{n^2-1}{n^2}\right)^{n-\frac 1 2} \sqrt{\frac n {n-1}} \left( \frac{1 - \sqrt{\frac 2 3 \frac{n-2}{n-\frac 1 2}}}
{(1-\sqrt{\frac 1 3 \frac{n-2} n})(1 + \frac 1 n \sqrt{\frac 2 3 \frac{n-2}{n-\frac 1 2}})} \right)^{n-1} \ . $$
Let $f_6(n) :=  \frac{1 - \sqrt{\frac 2 3 \frac{n-2}{n-\frac 1 2}}}{1-\sqrt{\frac 1 3 \frac{n-2} n}}$. Then $f_6(5) < 0.604$ and we claim that $f_6(n) < \frac 3 5$ for all $n \ge 6$. This is equivalent to $2 < 5 \sqrt{\frac 2 3 \frac{n-2}{n-\frac 1 2}} - 3 \sqrt{\frac 1 3 \frac{n-2} n} =: f_7(n)$. Now
$\sqrt{\frac{n-2} {n- \frac 1 2}} \ge 1 - \frac 1 n$, $\sqrt{\frac{n-2} n} \le 1 - \frac 1 n$, so that $f_7(n) \ge (5 \sqrt{\frac 2 3} - \sqrt 3) (1 - \frac 1 n) >2$ for all $n \ge 7$. Also $f_7(6) >2$. Further, $f_8(n) := (1 + \frac 1 n \sqrt{\frac 2 3 \frac{n-2}{n-\frac 1 2}})^{n-1}$ is increasing in $n$, being essentially of the form $(1+ \sqrt{\frac 2 3} \frac 1 n (1+O(\frac 1 n)))^n$, with $f_8(n) \ge f_8(5) > 1.64$ for all $n \ge 5$. Hence for $n \ge 6$, we have
$Q_n(t_2) < \frac{2 \sqrt 2 n + \frac 1 {2 \sqrt 2}}{1.64} \left(\frac 3 5 \right)^{n-1} =: f_9(n)$, which is decreasing in $n$, with $f_9(n) \le f_9(6) < 0.83 < 1$ for $n \ge 6$. For $n=5$, the estimate has to be refined: first, replace $t_2 \simeq 0.6124$ by $\tilde{t_2} = \sqrt{\frac{n- \frac 2 5}{2(n+1)}} \simeq 0.6191$, with $\frac 1 {2(\tilde{t_2} - t_1)} < 12$ and the corresponding $\tilde{f_6}(5) < 0.617$, $\tilde{f_8}(5) > 1.64$. We also use $f_2(5) < 0.931$. Then for $n=5$, $Q_n(\tilde{t_2}) < \frac{12}{1.64} \ 0.931 \cdot 0.617^4 < 0.99 < 1$. This leaves the estimate for $Q_n(a_1)$ for $a_1 \in [t_2,\tilde{t_2}]$, $n=5$. In this situation, we may improve the estimate $1 + \frac 1 n \frac{\psi(a_1)}{a_1} \ge 1 + \frac 1 n \frac t {a_1}$ leading to $L$, by using $\frac{\psi(a_1)}{a_1} \ge \frac{\psi(\tilde{t_2})}{\tilde{t_2}} > 0.86$ and $(1+ \frac 1 n \frac{\psi(a_1)}{a_1})^{n-1} > (1+ \frac{0.86} n)^{n-1} > 1.88$, $n=5$. Then for $a_1 \in [t_2,\tilde{t_2}]$ and $n=5$, $\frac L R < \frac{10 \sqrt 2 + \frac 1 {2 \sqrt 2}} {1.88} \ 0.931 \cdot 0.604^4 < 0.96 < 1$. This shows $A(a,t) < A(a^{(1)},t)$ also for $n=5$. \\

\vspace{0,5cm}

vi) We now study the case $n \ge 5$ and $t_0 < t < a_2 = \psi(a_1) < t_1 < a_1 < t_2$. Again, we estimate the right side $B(a_1,a_2,t)$ in \eqref{eq4.2}. Let $\gamma := \frac{a_1+a_2}{n-1}$ and consider $g(x) := (\frac{x-t}{x+\gamma})^{n-1}$. Then
$g'(x) = (n-1) \frac{(x-t)^{n-2}}{(x+\gamma)^n} (t + \gamma)$, $g''(x) = (n-1) \frac{(x-t)^{n-3}}{(x+\gamma)^{n+1}} (t + \gamma) (n(t+\gamma)-2(x+\gamma))$. Since $4 t > 2 a_1$, $2 t > a_2$, $g''(x) > 0$ for all $x \in (a_2,a_1)$. Therefore $g'$ is increasing in $(a_2,a_1)$ and we find from \eqref{eq4.2} with suitable $\theta \in (a_2,a_1)$
$$B(a_1,a_2,t) = g'(\theta) \le g'(a_1) = (n-1) \frac{(a_1-t)^{n-2}}{(a_1+\frac{a_1+a_2}{n-1})^n} \left(t + \frac{a_1+a_2}{n-1} \right)  \ . $$
To estimate $B(a_1,a_2,t)$, we will use $a_1+a_2 = a_1 + \psi(a_1) \le 2 t_1$. This follows from $\psi'(x) \le -1$, $x \ge t_1$.
Further, we claim that $\frac{a_2}{a_1} \ge 1 - \frac 3 5 \frac 1 n$. Since $\psi$ is decreasing,
$\frac{a_2}{a_1} \ge \frac{\psi(t_2)}{t_2} = \frac{\sqrt{n^2-1}} n \sqrt{\frac{2n+1}{2n-1}} - \frac 1 n$. Our claim will follow from
$\frac{\sqrt{n^2-1}} n \sqrt{\frac{2n+1}{2n-1}} \ge 1 + \frac 2 5 \frac 1 n$. This is true, since for $n \ge 5$ we have
$\frac{n^2-1} n \frac{2n+1}{2n-1} - 1 \ge \frac{13}{15} \frac 1 n \ge \frac 4 5 \frac 1 n + \frac 4 {25} \frac 1 {n^2}$. Thus
$(1 + \frac 1 n \frac{a_2}{a_1})^n \ge (1+ \frac 1 n (1- \frac 3 5 \frac 1 n))^n \ge (1+ \frac{22}{25} \frac 1 n)^n > \frac{20} 9$ for $n \ge 5$. We find
\begin{align*}
B(a_1,a_2,t) & \le (n-1) \left(\frac{n-1} n \right)^n \left(t + \frac{2 t_1}{n-1} \right) \frac{(1- \frac t {a_1})^{n-2}}{a_1^2 (1 + \frac 1 n \frac{a_2}{a_1})^n} \\
& \le \frac 9 {20} (n-1) \left(\frac{n-1} n \right)^n \left(t + \frac{2 t_1}{n-1} \right) \frac{(1- \frac t {a_1})^{n-2}}{a_1^2} \ .
\end{align*}
The function $k(a_1) := \frac{(1- \frac t {a_1})^{n-2}}{a_1^2}$ is increasing in $a_1$, since $(\ln k)'(a_1) = \frac{n t - 2 a_1}{a_1 (a_1-t)} > 0$. Therefore
$$B(a_1,a_2,t) \le \frac 9 {20} (n-1) \left(\frac{n-1} n \right)^n \left(t + \frac{2 t_1}{n-1} \right) \frac{(1- \frac t {t_2})^{n-2}}{t_2^2} =: L \ . $$
Again we claim that $A(a,t) < A(a^{(1)},t)$. This will follow from $\frac L R < 1$ with $R$ as in v). We have
$$\frac L R = \frac 9 {20} (n-1) \left(\frac{n^2-1}{n^2} \right)^n \sqrt{\frac n {n-1}} \frac{t + \frac{2 t_1}{n-1}}{t_2^2} \frac{(1-\frac t {t_2})^{n-2}}{(1-\sqrt{\frac{n+1} n} t)^{n-1}} \ . $$
The function $h(t) := (t + \delta) \frac{(1-\frac t {a_1})^{n-2}}{(1-\sqrt{\frac{n+1} n} t)^{n-1}}$ is decreasing in $t$, for any fixed $\delta \in (0,1)$, since $(\ln h)'(t) =  \frac 1 {t+\delta} - \frac{n-2}{a_1-t} + \frac{n-1}{\sqrt{\frac n {n+1}} - t} = \frac 1 {t+\delta} + \frac 1{\sqrt{\frac n {n+1}} - t} -(n-2) \frac{\sqrt{\frac n {n+1}} -a_1} {(a_1-t)(\sqrt{\frac n {n+1}} - t)} < 0$. Thus $\frac L R$ will be maximal for  $t=t_0$, when
$\frac{t_0}{t_2}  = \sqrt{\frac 2 3 \frac{n-2}{n- \frac 1 2} }$, $\sqrt{\frac{n+1} n} t_0  = \sqrt{\frac 1 3 \frac{n-2} n}$ and then, also using
$(\frac{n^2-1}{n^2} )^n < \exp(-\frac 1 n)$,
$$\frac L R \le \frac 9 {20} (n-1) \exp(-\frac 1 n) \left( \frac 2 {\sqrt 3} \frac{\sqrt{n(n-2)}}{n- \frac 1 2} + 2 \sqrt 2 \frac{\sqrt{\frac n {n-1}}}{n-\frac 1 2} \right) \frac{(1-\sqrt{\frac 2 3 \frac{n-2}{n- \frac 1 2}})^{n-2}}{(1-\sqrt{\frac 1 3 \frac{n-2} n})^{n-1}} =: \alpha(n) \ . $$
This gives $\frac L R \le \alpha(5) < 0.99 < 1$ for $n=5$. For $n > 5$, we use $\sqrt{1+x} \le 1 + \frac x 2$ for $|x| < 1$ to estimate
$\frac 2 {\sqrt 3} \frac{\sqrt{n(n-2)}}{n- \frac 1 2} + 2 \sqrt 2 \frac{\sqrt{\frac n {n-1}}}{n-\frac 1 2}  \le \frac 2 {\sqrt 3} - \frac 1 {\sqrt 3 n} + \frac{2 \sqrt 2}{n-1} < \frac 8 5$ for all $n \ge 7$. For $n=6$, the left side is also $< \frac 8 5$. Therefore
$$\frac L R \le \frac{18}{25} (n-1) \exp(- \frac 1 n)  \frac{(1-\sqrt{\frac 2 3 \frac{n-2}{n- \frac 1 2}})^{n-2}}{(1-\sqrt{\frac 1 3 \frac{n-2} n})^{n-1}} \ . $$
We know from part v) that  $\frac{1-\sqrt{\frac 2 3 \frac{n-2}{n- \frac 1 2}}}{1-\sqrt{\frac 1 3 \frac{n-2} n}} < \frac 3 5$ for $n \ge 6$. Using also
$1 - \sqrt{\frac 1 3 \frac{n-2} n} \ge 1 - \frac 1 {\sqrt 3} > \frac {21}{50}$, we get for all $n \ge 6$ that
$\frac L R \le \frac {12} 7 \exp(-\frac 1 n) (n-1) \left(\frac 3 5 \right)^{n-2} =: \beta(n)$. Since $(\ln \beta)'(n) = \frac 1 {n^2} + \frac 1 {n-1} - \ln(\frac 5 3) < 0$ for all $n \ge 6$, $\beta$ is decreasing with $\frac L R \le \beta(n) \le \beta(6) < 0.95 < 1$ for all $n \ge 6$. Therefore $A(a,t) < A(a^{(1)},t)$ for all $n \ge 5$. \\

\vspace{0,5cm}

vii) Last we consider the dimension $n=4$ and $a_1 > t_1 = \sqrt{ \frac 3 {10}}$, $a_2 = \psi(a_1) > t > t_0 = \sqrt{\frac 2 {15}}$. We have
\begin{equation}\label{eq4.4}
\frac 6 {\sqrt 5} A(a,t) = \frac 1 {a_1-a_2} \left( \frac{(a_1-t)^3}{(a_1+\frac{a_1+a_2} 3)^3} - \frac{(a_2-t)^3}{(a_2+\frac{a_1+a_2} 3)^3 } \right) \ .
\end{equation}
This is equal to $\frac 1 {a_1-a_2} \int_{a_2}^{a_1} \frac{(x-t)^2}{(x+\alpha)^4} \ dx \ (t + \alpha)$, $\alpha = \frac{a_1+a_2} 3 \ge \beta := 0.3548$, since
$h(x) = \frac{(x-t)^3}{(x+\alpha)^3}$ has the derivative $h'(x) = \frac{(x-t)^2}{(x+\alpha)^4} (t + \alpha)$. Let $t_0 := 0.404$. We show that
$\frac {A(a,t)}{A(a^{(1)},t)}$ is decreasing for $t \in [t_0,\sqrt{\frac 3 {10}}]$, where $\frac 6 {\sqrt 5} A(a^{(1)},t) = \frac {16}{25} \left(\sqrt{\frac 4 5} - t \right)^3$. By the integral formula, it suffices to prove that $k(t) := \frac {(x-t)^2}{(\sqrt{\frac 4 5} - t)^3} (t+ \beta)$ is decreasing in $t$ for any $x \le a_1 \le \psi(t)$. Note here that $t \le a_2 = \psi(a_1)$ implies $a_1 \le \psi(t)$ since $\psi^2 = \Id$ and $\psi$ is decreasing. We have for $x \le \psi(t_0) \le 0.6717$
\begin{align*}
(\ln k)'(t) & = \frac{-2}{x-t} + \frac 3 {\sqrt{\frac 4 5} - t} + \frac 1 {t + \beta} \\
& = - \frac{2 \sqrt{\frac 4 5} +t -3 x}{(x-t) \left(\sqrt{\frac 4 5} - t \right)} + \frac 1 {t + \beta}
\le - \frac{0.1777}{0.268 \cdot 0.491} + \frac 1 {0.758} \le - 0.03 < 0 \ .
\end{align*}
Therefore $k$ is decreasing for $t \in [t_0,\sqrt{\frac 3 {10}}]$. Thus if we prove that $A(a,t_0) \le A(a^{(1)},t_0)$ holds, $A(a,t) \le A(a^{(1)},t)$ will follow for all $t \in [t_0,\sqrt{\frac 3 {10}}]$. \\
Calculation shows that the formal singularity in \eqref{eq4.4} can be removed; we find with $a_2 = \psi(a_1)$
\begin{align}\label{eq4.5}
\frac 6 {\sqrt 5} A(a,t) & = \left(\frac{t + \frac{a_1+a_2} 3}{(a_1 + \frac{a_1+a_2} 3)(a_2 + \frac{a_1+a_2} 3)} \right)^3 (a_1-a_2)^2 \nonumber \\
& \quad \quad + \frac{3 (t + \frac{a_1+a_2} 3)}{(a_1 + \frac{a_1+a_2} 3)^2 (a_2 + \frac{a_1+a_2} 3)^2} (a_1-t) (a_2-t) \\ \nonumber
& =: f_1(t,a_1,a_2) + f_2(t,a_1,a_2) =: f(t,a_1,a_2) \ . \nonumber
\end{align}
We claim that $f(t,a_1,\psi(a_1))$ is decreasing in $a_1 \in [\sqrt{\frac 3 {10}}, \psi(\sqrt{\frac 3 {10}})] \simeq [0.5477,0.6993]$ for all $t \in [\sqrt{\frac 2 {15}},t_0]$. Then the maximum for these $t$ in \eqref{eq4.5} will be attained for $a_1 = a_2 = \sqrt{\frac 3 {10}}$, when $f_1 = 0$ and $f_2 = \frac 6 {\sqrt 5} A(a^{(2)},t)$. Therefore for $t \in [\sqrt{\frac 2 {15}},t_0]$, $A(a,t) \le A(a^{(2)},t)$, with equality for $a_1 = \sqrt{\frac 3 {10}}$. For $t > \tilde{t} \simeq 0.3877$, $A(a^{(2)},t) < A(a^{(1)},t)$, so that, in particular, $A(a,t_0) < A(a^{(1)},t_0)$ holds as required above. \\

Actually, for all $t \in [\sqrt{\frac 2 {15}},\sqrt{\frac 3 {10}}]$, $f_2(t,a_1,\psi(a_1))$ is decreasing in $a_1 \in [\sqrt{\frac 3 {10}},\psi(t)]$ and
$f_1(t,a_1,\psi(a_1))$ is increasing. For $t \in [\sqrt{\frac 2 {15}}, t_0]$, $f_2$ is decreasing faster than $f_1$ is increasing (up from 0). One has for $a_1$  near $\sqrt{\frac 3 {10}}$
\begin{align*}
f&(t, a_1,\psi(a_1)) = \frac 6 {\sqrt 5} A(a^{(2)},t) \\
& - \frac{54}{125} \left(4 t - \sqrt{\frac 2 {15}} \right) \left(\frac 7 3 + 10 \sqrt{\frac 2 {15}} t - 20 t^2 \right) \left(a_1- \sqrt{\frac 3 {10}} \right)^2 +O \left( \left(a_1- \sqrt{\frac 3 {10}} \right)^3 \right) \ .
\end{align*}
This means that close to $\sqrt{\frac 3 {10}}$, $f$ is decreasing for all $t$ with $t \le \frac 1 4 (\sqrt 2 + \sqrt{\frac 2 {15}}) \simeq 0.4448$. For $t \le t_0$, $f$ is decreasing in the full interval $[\sqrt{\frac 3 {10}},\psi(\sqrt{\frac 3 {10}})]$, as a tedious calculation of the derivative
$\frac d {d a_1} f(t,a_1,\psi(a_1))$ shows, using that $\psi(a_1) \le -1$ for $a_1 \in [\sqrt{\frac 3 {10}},\psi(\sqrt{\frac 3 {10}})]$. We do not give the details of the calculation. The conclusion is $A(a,t) \le \max ( A(a^{(1)},t) , A(a^{(2)},t) )$, with equality for $t \le \tilde{t}$ and $a_1=a_2$ and strict inequality for $t> \tilde{t}$.
\end{proof}

{\it Remark.} For $0.411 < t < 0.444$, $f$ in part vii) is first decreasing and then increasing in $[\sqrt{\frac 3 {10}},\psi(t)]$. For $t > 0.445$ $f$ is increasing in the full interval. The situation for $n=4$ is slightly different than for $n \ge 5$, in the sense that in the second case $a_1 > a_2 >t$, $t \le \tilde{t}$ leads to a maximal situation $A(a,t)=A(a^{(2)},t)$ in the limit $a_1=a_2 = \sqrt{\frac{n-1}{2 (n+1)}}$ for $n=4$, whereas for $n \ge 5$ it does not. The above quotient $\frac{A(a,t)}{A(a^{(1)},t)}$ is actually also decreasing for $t \in[\sqrt{\frac 2 {15}},t_0]$, but the quotient may be $>1$ for $t < \tilde{t}$. \\

\vspace{0,5cm}

\section{Dimensions $n=2$ and $n=3$}

We now prove Propositions \ref{prop2} and \ref{prop3} and start with the easy case of dimension 2. \\

{\it Proof} of Proposition \ref{prop2}. \\
Let $a \in S^2 \subset \R^3$ with $\sum_{j=1}^3 a_j = 0$, $t < \sqrt{\frac 2 3}$ and $a_1 > t$. Then $a_2+a_3 = -a_1$, $a_2^2+a_3^2 = 1-a_1^2$ imply that
$a_{2,3} = -\frac{a_1} 2 \pm \frac 1 2 \sqrt{2-3 a_1^2}$, $(a_1-a_2)(a_1-a_3) = 3 a_1^2-\frac 1 2$. By Proposition \ref{prop2.1}
$A(a,t) = 2 \sqrt 3 \frac{a_1-t}{6 a_1^2-1}$. \\
If $e_2, e_3 \in H_-(a,t) = \{ x \in \Delta^2 | \langle a , x \rangle < t \}$, we have by Lemma \ref{lem1} that $2 t > a_2 + a_3 = - a_1 \ge - \sqrt{\frac 2 3}$, $t > - \frac 1 {\sqrt 6}$. Therefore, if $t \le - \frac 1 {\sqrt 6}$, only one vertex $e_j$ can be in $H_-(a,t)$. Using $A(a,t) = A(-a,-t)$, it suffices to consider only one non-zero term in formula \eqref{eq2.1} for $A(a,t)$ when
i) $\frac 1 {\sqrt 6} \le t < \sqrt{\frac 2 3}$ or ii) $t \in [0,\frac 1 {\sqrt 6}]$ or $-t \in [0,\frac 1 {\sqrt 6}]$. \\
We have $\frac d {d a_1} \frac{a_1-t}{6 a_1^2-1} = - \frac{6a_1^2 + 1 - 12 a_1 t}{(6 a_1^2 -1 )^2}$. This is zero for $a_1 = t + \sqrt{t^2-1/6}$, recalling $a_1 >t$. For $t > \frac 5 4 \frac 1 {\sqrt 6}$, $a_1 > \sqrt{\frac 2 3}$ which is impossible. For  $t > \frac 5 4 \frac 1 {\sqrt 6}$, $\frac d {d a_1} A(\cdot,t) < 0$, and $A(a,t)$ is maximal for $a_1 = \sqrt{\frac 2 3}$, $a_2 = a_3 = - \frac 1 {2 \sqrt 3}$, i.e. $a = a^{(1)}$. For $\frac 1 {\sqrt 6} < t < \frac 5 4 \frac 1 {\sqrt 6}$, $\frac 1 {\sqrt 6} < a_1:= t + \sqrt{t^2-1/6} < \sqrt{\frac 2 3}$, and $\frac{x-t}{6 x^2-1}$ is increasing for $x < a_1$ and decreasing for $x>a_1$. Thus $a^{\{t\}} \in S^2$ yields the maximum of $A(\cdot,t)$ in this range of $t$. For $0 \le t \le \frac 1 {\sqrt 6}$, again $\frac d {d a_1} A(\cdot,t) < 0$ and $A(\cdot,t)$ attains its minimum for the maximal value of $a_1$, i.e. for $a^{(1)}$. \\

As for the maximum of $A(\cdot,t)$ for $0 \le t \le \frac 1 {\sqrt 6}$, we have to consider $-t$, i.e. $A(a,t) = 2 {\sqrt 3} \frac{a_1+t}{6 a_1^2 -1}$. Then $\frac d {d a_1} \frac{a_1+t}{6 a_1^2-1} = - \frac{6a_1^2+1+12 a_1 t}{(6a_1^2-1)^2} < 0$ and $A(\cdot,t)$ is decreasing in $a_1$. We need $a_1 > -t \ge a_2,a_3$. Thus $a_2+t \le 0$ with $a_2 = \frac 1 2 (\sqrt{2-3 a_1^2} -a_1)$. This implies $a_1 \ge \frac 1 2 (\sqrt{2-3 t^2} + t)$, and in the maximum case $a_1 = \frac 1 2 (\sqrt{2-3 t^2} + t)$, $a_2 = -t$ and $a_3 = -\frac 1 2 (\sqrt{2 - 3 t^2} - t)$: Replacing $-t = |t|$ again by $t$ yields the maximum $a^{[t]}$ stated in Proposition \ref{prop2}. \hfill $\Box$

\vspace{0,5cm}

In dimension 3 there are more critical points of $A(a,t)$. \\

{\it Proof} of Proposition \ref{prop3}. \\
Let $a \in S^3 \subset \R^4$ with $\sum_{j=1}^4 a_j = 0$ and $t \le \frac{\sqrt 3} 2$. If $e_2, e_3, e_4 \in H_-(a,t)$, we get using Lemma \ref{lem1} i) that
$3 t > a_2+a_3+a_4 = - a_1 \ge - \frac{\sqrt 3} 2$, $t > - \frac 1 {2 \sqrt 3}$. Therefore, if $t \le - \frac 1 {2 \sqrt 3}$, only one or two vertices of $\Delta^3$ can be in $H_-(a,t)$. Using $A(a,t) = A(-a,-t)$, it suffices to consider one or two non-zero terms in formula \eqref{eq2.1} for $A(a,t)$ when
i) $\frac 1 {2 \sqrt 3} \le t < \frac{\sqrt 3} 2$ or ii) $t \in [0,\frac 1 {2 \sqrt 3}]$ or $-t \in [0,\frac 1 {2 \sqrt 3}]$. Thus we have to consider
two possible cases in the volume formula \eqref{eq2.1}: if $a_1 > t \ge a_2, a_3, a_4$, there is only one non-zero term in \eqref{eq2.1}; if $a_1> a_2 > t \ge a_3, a_4$, there are two terms, exchanging possibly $t$ and $-t$ when $|t| \le \frac 1 {2 \sqrt 3}$. \\

a) i) Consider first the case $a_1 > t \ge a_2, a_3, a_4$. As in the proof of Theorem \ref{th1}, there are at most two different coordinates among $(a_2,a_3,a_4)$. Letting $a_2 = a_3 =: c$, $a_4 =: d$, we have $a_1 + 2 c + d =0$, $a_1^2 + 2 c^2 + d^2 =1$. We have two possibilities
$$ \left\{\begin{array}{c@{\quad}l}
c_+ = & -\frac 1 3 a_1 + \frac 1 6 \sqrt{6-8a_1^2} \\
d_+ = & -\frac 1 3 a_1 - \frac 1 3 \sqrt{6-8a_1^2}
\end{array}\right\}  \quad , \quad
 \left\{\begin{array}{c@{\quad}l}
c_- = & -\frac 1 3 a_1 - \frac 1 6 \sqrt{6-8a_1^2}    \\
d_- = & -\frac 1 3 a_1 + \frac 1 3 \sqrt{6-8a_1^2}
\end{array}\right\}  \ . $$
For $(c_+,d_+)$, $c_+ < a_1$ requires $a_1 > \frac 1 {2 \sqrt 3}$ and if $a_1 > t \ge \frac 1 {2 \sqrt 3}$, $c_+ < t$ is satisfied. If $0 \le t < \frac 1 {2 \sqrt 3}$, for $c_+ < t$ we need $a_1 > \frac 1 2 \sqrt{2-8 t^2}-t$. \\
For $(c_-,d_-)$, $d_- < a_1$ requires $a_1 > \frac 1 2$ and if $a_1 > t \ge \frac 1 2$, $d_- < t$ is satisfied. If $0 \le t \le \frac 1 2$, for $d_- < t$ we need $a_1 > \frac 1 3 \sqrt{6-8 t^2} - \frac t 3$. \\
Proposition \ref{prop2.1} yields  $A(a,t) = f_{\pm}(a_1,t) = \frac{(a_1-t)^2}{(a_1-c_{\pm})^2 (a_1-d_{\pm})} $. We find with $' = \frac d {da_1}$ \\
$(\ln f_{\pm})'(a_1,t) = \frac 4 {a_1-t} \frac 1 {4a_1 \pm \sqrt{6-8 a_1^2}} \frac 1 {8a_1 \mp \sqrt{6-8a_1^2}} [ \pm(3t-a_1) \sqrt{6-8a_1^2} + 30 a_1 t - 10 a_1^2 -3] \ .$
The denominators are positive, if $a_1 > \frac 1 {2 \sqrt 3}$ or $a_1 > \frac 1 2$, respectively. \\

In the case of $f_+$, let $t_0:= \frac 7 {10} \frac 1 {\sqrt 3} \simeq 0.4041$. Since the factor \\
$F(a_1,t) = (3t-a_1) \sqrt{6-8a_1^2} + 30 a_1 t - 10 a_1^2 -3$ is increasing in $t$, its minimum for $t \ge t_0$ is in $t_0$, with $F(a_1,t_0) = (\frac 7 {10} \sqrt 3 -a_1) \sqrt{6-8a_1^2} + 7 \sqrt 3 a_1 - 10 a_1^2 -3$, which is decreasing in $a_1$, with value $0$ in $a_1 = \frac{\sqrt 3} 2$. Therefore $(\ln f_+)' \ge 0$ for all $\frac 1 {2 \sqrt 3} \le a_1 \le \frac{\sqrt 3} 2$ and $t \ge t_0$. Hence for $t \ge t_0$, the maximum of $f_+$ is in $a_1 = \frac {\sqrt 3} 2$, i.e. $a^{(1)}$ attains the maximum. The minimum here is irrelevant, since $t\ge t_0 > \frac 1 {2 \sqrt 3}$ and the the minimum of $A(\cdot,t)$ is zero. \\

For $\frac 1 {2\sqrt 3} < t < \frac 7 {10} \frac 1{\sqrt 3} = t_0$, $F(a_1,t)=0$ if and only if $t= \frac{a_1} 3 + \frac 1 {10a_1+\sqrt{6-8 a_1^2}} =: \phi(a_1)$. The function $\phi$ is strictly increasing and bijective $\phi: [\frac 1 {2\sqrt 3},\frac{\sqrt 3} 2] \to [\frac 1 {2\sqrt 3},\frac 7 {10} \frac 1{\sqrt 3}]$. For $\frac 1 {2\sqrt 3} < t < \frac 7 {10} \frac 1{\sqrt 3}$, let $a_0$ be the unique solution of $t= \phi(a_1)$. Then $f_+$ is increasing for $a_1 < a_0$ and decreasing for $a_0< a_1$, hence attains its maximum at $a_0$: In this case, $A(\cdot,t)$ attains its maximum at $\tilde{a} :=(a_0,c(a_0),c(a_0),d(a_0))$,
$A(\tilde{a},t) = \frac{(a_0-t)^2}{(a_0-c(a_0))^2 (a_0-d(a_0))} = \frac{(a_0- \phi(a_0))^2}{(a_0-c(a_0))^2 (a_0-d(a_0))} =: \Phi(a_0)$. It can be checked that for $a_0 \ge \frac 1 3$, this is strictly smaller than $A(a^{(2)},t) = A(a^{(2)},\phi(a_0)) = \frac 1 2 - 2 \phi(a_0)^2$, so that $\tilde{a}$ does not yield the absolute maximum of $A(\cdot,t)$ for $t \ge \phi(\frac 1 3) = \frac{22+\sqrt {19}}{90} \simeq 0.2928$. However, for $t$ very close to $\frac 1 {2 \sqrt 3} \simeq 0.2887$, $A(\tilde{a},t) > A(a^{(2)},t)$. But in this case $A(\tilde{a},t) < A(a^{\{t\}},t)$ for the vector $a^{\{t\}}$ given in Proposition \ref{prop3} and considered in part b) i) below. Actually, $\Phi$ is decreasing in $a_0$ and $A(a^{\{t\}},t)$ is decreasing in $t$, so that for $\frac 1 {2 \sqrt 3} < t < \phi(\frac 1 3) \simeq 0.2928$,
$A(\tilde{a},t)  = \Phi(a_0) < \Phi(\frac 1 {2\sqrt 3}) = \frac {16}{75} \sqrt 3 \simeq 0.3695 < 0.4069 \simeq A(a^{\{(\phi(\frac 1 3)\}},\phi(\frac 1 3)) \le A(a^{\{t\}},t)$,
and again $\tilde{a}$ does not yield the maximum of $A(\cdot,t)$. We have $\lim_{s \searrow \frac 1 {2 \sqrt 3}} A(a^{\{s\}},s) = \frac 5 {18} \sqrt 3$, see part b) i) below. \\

For  $0 \le t < \frac 1 {2 \sqrt 3}$, $F(a_1,t) <0$ : since it is increasing in $t$, the maximum of $F$ is attained at $t= \frac 1 {2 \sqrt 3}$ and
$F(a_1,\frac 1 {2 \sqrt 3}) = (\frac {\sqrt 3} 2  -a_1) \sqrt{6-8a_1^2} + 5 \sqrt 3 a_1 - 10 a_1^2 -3$ is decreasing in $a_1$ and $0$ in $a_1 = \frac 1 {2 \sqrt 3}$. Thus $(\ln f_+)' \le 0$ for all $\frac 1 {2 \sqrt 3} \le a_1 \le \frac{\sqrt 3} 2$ and $0 \le t \le \frac 1 {2 \sqrt 3}$. Hence the minimum of $A = f_+$ for all $0 \le t \le \frac 1 {2 \sqrt 3}$ is in $a_1 = \frac {\sqrt 3} 2$, i.e. $a^{(1)}$ attains the minimum of $A(\cdot,t)$. The maximum of $A(\cdot,t)= f_+$ is attained at the minimal possible value $a_1 = \frac 1 2 \sqrt{2-8 a_1^2} - t$, and the extremal vector is
$$a^{[t]} = \left(\frac 1 2 \sqrt{2-8 a_1^2} - t, t, t, -\frac 1 2 \sqrt{2-8 a_1^2} - t \right) $$
with $A(a^{[t]},t) = \frac 1 {\sqrt{2 - 8 t^2}}$. \\

ii) In the case of $f_-$, let $F(a,t) = (a_1 - 3 t) \sqrt{6-8a_1^2} + 30 a_1 t - 10 a_1^2 -3$. Recall $a_1 > \frac 1 2$. Then
$\frac d {dt} F(a_1,t) = 30 a_1 -3 \sqrt{6 - 8 a_1^2} > 0$ for $a_1 > \frac 1 2$. Hence $F$ is increasing in $t$. For $0 \le t \le t_0 = \frac 7 {10} \frac 1 {\sqrt 3}$, $f$ is maximal in $t_0$, with $F(a_1,t_0) = (a_1 - \frac 7 {10} \sqrt 3) \sqrt{6-8a_1^2} + 7 \sqrt 3 a_1 - 10 a_1^2 -3$, which is increasing in $a_1$, $F(\frac {\sqrt 3} 2 , t_0) = 0$. Therefore $(\ln f_-)' \le 0$ for all $0 \le t \le t_0$ and $\frac 1 2 \le a_1 \le \frac{\sqrt 3} 2$, and $f_-$ attains its maximum at $a_1 = \frac 1 2$, i.e. $a^{(2)} = \frac 1 2 (1,-1,-1,1)$ attains the maximum of $A(\cdot,t)$ in the case of $f_-$ for $0 \le t \le t_0$. The minimum is in $a_1 = \frac{\sqrt 3} 2$, with vector $a^{(1)}$. \\

For $t_0 < t < \frac 1 2$, we have $F < 0$ in $\frac 1 2 < a_1 < a_0$ and $F>0$ in $a_0 < a_1 < \frac {\sqrt 3} 2$, where $a_0$ is the solution of
$t = \frac{3+10 a_1^2 - a -\sqrt{6-8 a_1^2}}{3(10 a_1 - \sqrt{6-8 a_1^2})} =:\phi(a_1)$. Here $\phi : [\frac 1 2, \frac{\sqrt 3} 2] \to [t_0,\frac 1 2]$ is bijective and strictly decreasing, so that there is a unique solution $a_1 \in [\frac 1 2, \frac{\sqrt 3} 2]$. The minimal possible value of $a_1$ is
$\tilde{a} = \frac 1 3 \sqrt{6 - 8 t^2} - \frac t 3$. Moreover, $\phi(\tilde{a}) > t = \phi(a_0)$. Since $\phi$ is decreasing, $\tilde{a} < a_0$. Hence the maximum of $f_-$ is attained either in $a_1=\tilde{a}$ or in $a_1=\frac{\sqrt{3}} 2$. For $t \le \bar{t} \simeq 0.4259$ this occurs in $\tilde{a}$, for $t \ge \bar{t}$ in $\frac{\sqrt{3}} 2$. In the latter case, the maximum of $A(\cdot,t)$ is in $a^{(1)}$. However, for $t_0 < t < \bar{t}$,
$f_-(\tilde{a},t) < \frac 1 2 - 2 t^2 = A(a^{(2)},t)$, so that the maximum of $f_-$ for $t_0 < t < \frac 1 2$ is attained either in $a^{(1)}$ or in $a^{(2)}$. \\

b) i) Next we consider the second case $a_1 > a_2 > t \ge a_3, a_4$. Then $a_3+a_4 = -(a_1+a_2)$, $a_3^2+a_4^2= 1 -a_1^2-a_2^2$ implies
$a_{3,4} = -\frac{a_1+a_2} 2 \pm \frac 1 2 \sqrt{2 (1 -a_1^2-a_2^2)-(a_1+a_2)^2}$, where $2 (1 -a_1^2-a_2^2)-(a_1+a_2)^2 \ge 0$ requires
$a_1 \le \frac 1 3 \sqrt{6-8 a_2^2}- \frac{a_2} 3 =: \psi(a_2) \le \sqrt{\frac 2 3}$. We also have $t < a_2 \le \frac 1 2 \ ( = \sqrt{\frac{n-1}{2(n+1)}} )$, as shown in Lemma \ref{lem1}. Therefore $t < a_2 \le \frac 1 2$, $a_1 \le \sqrt{\frac 2 3}$. Proposition \ref{prop2.1} yields with $(a_1-a_3)(a_1-a_4) = 3 a_1^2 + 2a_1 a_2 + a_2^2 - \frac 1 2$, $(a_2-a_3)(a_2-a_4) = a_1^2 + 2a_1 a_2 + 3 a_2^2 - \frac 1 2$ that
\begin{equation}\label{eq5.1}
A(a,t) = f(a_1,a_2,t) := \frac 1 {a_1-a_2} \Big( \frac{(a_1-t)^2}{3a_1^2+2a_1a_2+a_2^2- \frac 1 2} - \frac{(a_2-t)^2}{a_1^2+2a_1a_2+3a_2^2- \frac 1 2} \Big) \ .
\end{equation}
The critical points of $A(\cdot,t)$ satisfy $\frac{\partial f}{\partial a_1} = \frac{\partial f}{\partial a_2} = 0$. In particular, \\
$0 = (\frac{\partial f}{\partial a_1} - \frac{\partial f}{\partial a_2}) (a_1,a_2,t) =
\frac{(a_1-a_2) [4t(a_1+a_2)-(1-2a_1^2-2a_2^2)] [t(1-2a_1^2-2a_2^2)+2(a_1+a_2)^3-(a_1+a_2)]} {(3a_1^2+2a_1a_2+a_2^2- \frac 1 2)(a_1^2+2a_1a_2+3a_2^2- \frac 1 2)} \ . $
Therefore either $a_1=a_2$ (in the limit $a_2 \to a_1$) or
\begin{equation}\label{eq5.2}
t = \frac{1-2 (a_1+a_2)^2 + 4a_1a_2}{4(a_1+a_2)} \text{   or   }  t = \frac{(a_1+a_2)(1-2(a_1+a_2)^2)}{1- 2(a_1+a_2)^2 + 4a_1a_2} \ .
\end{equation}

First, suppose that $t \ge \frac 1 {2 \sqrt 3}$. Then there is no solution $(a_1,a_2)$ of \eqref{eq5.2}, since for $\frac 1 {2 \sqrt 3} \le t < a_2 < a_1 \le \sqrt{\frac 2 3}$ we have in the first case $\frac{1-2 (a_1+a_2)^2 + 4a_1a_2}{4(a_1+a_2)} < \frac 1 {2 \sqrt 3}$ because  \\
$a_1^2+a_2^2 + \frac{a_1+a_2}{\sqrt 3} - \frac 1 2 = (a_2-\frac 1 {2 \sqrt 3})(2a_2+\sqrt 3) + (2a_2+ \frac 1 {\sqrt 3})(a_1-a_2) + (a_1-a_2)^2 >0$
and in the second case also $\frac{(a_1+a_2)(1-2(a_1+a_2)^2)}{1- 2(a_1+a_2)^2 + 4a_1a_2} < \frac 1 {2 \sqrt 3}$, if $2a_1^2+2a_2^2 < 1$, with equality for $a_1=a_2 = \frac 1 {2 \sqrt 3}$, which is not attained since $a_1>a_2$. If $2a_1^2+2a_2^2 > 1$, the left side is $>3$ and thus cannot be equal to $t \le \sqrt{\frac 2 3}$, either. \\
Hence for $t \ge \frac 1 {2 \sqrt 3}$ the only critical points of $A(\cdot,t)$ occur for $a_1=a_2$ (limiting case). Then $a_1 = a_2 \le \frac 1 2$ and by l'Hospital's rule
\begin{equation}\label{eq5.3}
\lim_{a_2 \nearrow a_1} f(a_1,a_2,t) = \frac 4 {(12a_1^2-1)^2} \left(a_1 (8a_1^2-1) + t (1-4a_1^2) - 4t^2 a_1 \right) =: F(a_1,t) \ .
\end{equation}
For $a_1 = \frac 1 2$ we get $F(\frac 1 2,t) = \frac 1 2 - 2 t^2$, with $a = a^{(2)} = \frac 1 2(1,1,-1,-1)$. The critical points of F satisfy
$$\frac{\partial F}{\partial a_1} = \frac 4 {(12a_1^2-1)^3} \left[ 4(1+36a_1^2) t^2 -8a_1(5-12a_1^2) t + (1+12 a_1^2-96 a_1^4) \right] = 0$$
with solutions
\begin{equation}\label{eq5.4}
t_{\pm} = t_{\pm}(a_1) = \frac 1 {1+36a_1^2} \left[ a_1 (5 - 12 a_1^2) \pm \frac 1 2 (12 a_1^2-1) \sqrt{28a_1^2-1} \right] \ .
\end{equation}
Both functions $t_+$ and $t_-$ are bijective in the following ranges, $t_+ : [\frac 1 {2 \sqrt 3}, \frac 1 2] \to [\frac 1 {2 \sqrt 3}, \frac{\sqrt 6 + 1}{10}]$,
$t_- : [\frac 1 {2 \sqrt 3}, \frac 1 2] \to [-\frac{\sqrt 6 - 1}{10},\frac 1 {2 \sqrt 3}]$, where $t_+$ is increasing and $t_-$ decreasing. Here
$\frac{\sqrt 6 + 1}{10} \simeq 0.3449$, $-\frac{\sqrt 6 - 1}{10} \simeq -0.1449$. Then $t_{\pm}'(a_1)$ has the same sign as
$\pm a_1 (3024 a_1^4+ 504 a_1^2-31)-(216a_1^4+108 a_1^2-5/2) \sqrt{28 a_1^2-1}$. For $t_+$ this is non-negative and zero in $a_1= \frac 1 {2 \sqrt 3}$.
For $t > \frac 1 {2 \sqrt 3}$, \eqref{eq5.4} admits a solution $a_1$ only for $t_+= t_+(a_1)$. Moreover $t \le t_1 := \frac{\sqrt 6 + 1}{10}$ is required to have a solution. Hence the solution interval for $t = t_+(a_1)$ is $ t \in [\frac 1 {2 \sqrt 3},\frac{\sqrt 6 + 1}{10}]$, with $a_1 \in [\frac 1 {2 \sqrt 3},\frac 1 2]$. \\
For $t > t_1$, $\frac{\partial F}{\partial a_1} > 0$ and $F$ is increasing, with maximum in $a_1 = \frac 1 2$, i.e. for $a^{(2)}$: The absolute maximum of $A(\cdot,t)$ is then in $a^{(1)}$ or in $a^{(2)}$. Define $t_0 := \frac{9 + 4 \sqrt{16-6 \sqrt 3} }{2(3 \sqrt 3 + 16)} \simeq 0.4357$. For $t >t_0$, the maximum of $A(\cdot,t)$ is in $a^{(1)}$, for $t < t_0$, in $a^{(2)}$. \\
For $\frac 1 {2 \sqrt 3} < t < t_1$, solve $t = t_+(\bar{a_1})$ with $\bar{a_1} \in (\frac 1 {2 \sqrt 3},\frac 1 2)$. Then $\frac{\partial F}{\partial a_1}$ is positive for $\frac 1 {2 \sqrt 3} < a_1 < \bar{a_1}$ and negative for $\bar{a_1} < a_1 < \frac 1 2$. Hence $F(\cdot,t)$ attains its maximum in $\bar{a_1}$, with a larger value than in $a_1 = \frac 1 2$, i.e. larger than for $a^{(2)}$. This yields the maximum $a^{\{t\}}$ of $A(\cdot,t)$ given in Proposition \ref{prop3}. \\
For $t \searrow \frac 1 {2 \sqrt 3}$ we have $a_1 \searrow \frac 1 {2 \sqrt 3}$ and $F(a_1,t) \to \frac 5 9 \frac {\sqrt 3}  2 = \frac 5 9 A(a^{(1)},t)$. \\

ii) Secondly, we suppose that $0 \le t < \frac 1 {2 \sqrt 3}$ in the case $a_1>a_2>t \ge a_3,a_4$. Generally $a \in S^3$, $a_1 >0$, $\sum_{j=1}^4 a_j =0$ requires $a_1 \ge \frac 1 {2 \sqrt 3}$. Then $t_+(a_1) \ge \frac 1 {2 \sqrt 3}$, so that $t_+(a_1) = t$ is impossible. However, $t_-(a_1) = t$ has a solution as long as $t_-(a_1) \ge 0$. This means $a_1 \le \frac 1 4 \sqrt{1+ \sqrt {\frac {11} 3}} =: \tilde{a_1} \simeq 0.4268$. \\
But first, we consider the case $a_1 = a_2 \le \frac 1 2$. Then $t \ge a_3 = -a_1+\frac 1 2 \sqrt{2-8 a_1^2}$ is required. For $a_1 \in [\frac 1 {2 \sqrt 3}, \frac 1 2]$, $t_-(a_1) \ge -a_1 + \frac 1 2 \sqrt{2 - 8 a_1^2}$ is satisfied: The difference is zero in $a_1 = \frac 1 {2 \sqrt 3}$ and increasing in $a_1$. We have for the function $F$  given in \eqref{eq5.3} that
\begin{equation}\label{eq5.5}
F(a_1,t_-(a_1)) = \frac 4 {(1+36a_1^2)^2} \left[ a_1(5+52a_1^2) + (2 a_1^2+ \frac 1 2) \sqrt{28a_1^2-1} \right] \ .
\end{equation}
In part a) we had for $a^{[t]}$ that $A(a^{[t]},t) = \frac 1 {\sqrt{2-8 t^2}} =: \Phi(t)$ and then
$$\Phi(t_-(a_1)) = \frac {1+36a_1^2}{ 2 \sqrt{(1+4a_1^2)(1+36a_1^2)-2a_1(12a_1^2-1)(12a_1^2-5)(8a_1+\sqrt{28a_1^2-1}) } } \ . $$
One can check that $\Phi(t_-(a_1) > F(a_1,t_-(a_1))$ for all $a_1 \in (\frac 1 {2 \sqrt 3},\tilde{a_1}]$, with equality for $a_1 = \frac 1 {2 \sqrt 3}$. Therefore $a^{[t]}$ is also the maximum of $A(\cdot,t)$ in this case. \\

Next, we study the cases $t = \frac{1-2 (a_1+a_2)^2 + 4a_1a_2}{4(a_1+a_2)}$ and $t = \frac{(a_1+a_2)(1-2(a_1+a_2)^2)}{1- 2(a_1+a_2)^2 + 4a_1a_2}$ of \eqref{eq5.2} with $0 \le t = t_-(a_1) < \frac 1 {2 \sqrt 3}$. \\
We claim that the second case is impossible, since $a_3 \le t$ would not be satisfied: Let $x := a_1+a_2$, $y := 4a_1a_2$. Then $a_3 = -\frac x 2 + \frac 1 2 \sqrt{2-3 x^2 +y}$ and $t = \frac{x(1-2 x^2)}{1-2 x^2 +y}$, with $\frac 1 {2 \sqrt 3} \le x \le \frac 1 {\sqrt 2}$ and $0 \le y \le x^2 \le \frac 1 2$. Assume $a_3 \le t$. We show the contradiction $a_3 > t$ which is equivalent to
\begin{equation}\label{eq5.6}
(1-2 x^2 + y) (\sqrt{2 - 3 x^2 + y} -x) > 2 x (1 - 2 x^2) \ .
\end{equation}
We claim that $y > 2 x^2 + 2 x \sqrt{1 - 2 x^2} -1$. This will imply \eqref{eq5.6}: Since the left side of \eqref{eq5.6} is strictly increasing in $y$, \eqref{eq5.6} follows from  \\
$2 x \sqrt{1-2 x^2} (\sqrt{1-x^2+2x \sqrt{1-2x^2}}-x) \ge 2 x (1-2x^2)$ which is true with equality sign since
$(x+\sqrt{1-2x^2})^2 = 1-x^2+2x \sqrt{1-2x^2}$. \\
Now $y > 2 x^2 + 2 x \sqrt{1 - 2 x^2} -1$ is equivalent to
\begin{align*}
\left(1-2(a_1^2+a_2^2) \right)^2 & -4(a_1+a_2)^2 \left(1-2(a_1+a_2)^2 \right) \\
& = \left(6a_1^2-1+4a_1a_2+2a_2^2 \right) \left(2a_1^2-1+4a_1a_2+6a_2^2 \right) >0 \ ,
\end{align*}
which is true since $a_1 > a_2 > -\frac{a_1} 3 + \frac 1 6 \sqrt{6-8a_1^2}$ is satisfied: the assumption $a_2 > t \ge a_3 = -\frac{a_1+a_2} 2 + \frac 1 2 \sqrt{2 - 3(a_1^2+a_2^2)-(a_1+a_2)^2}$ yields $a_2 > -\frac{a_1} 3 + \frac 1 6 \sqrt{6-8a_1^2}$. Therefore the claim for $y$ is true and the contradiction $a_3 > t$ would follow. \\
This leaves the case $t = \frac{1-2 (a_1+a_2)^2 + 4a_1a_2}{4(a_1+a_2)}$. Calculation using \eqref{eq5.1} shows that then $A(a,t) = \frac 1 {2(a_1+a_2)}$. This is smaller than $A(a^{[t]},t) = \frac 1 {\sqrt{2-8 t^2}}$, since $2 - 8 t^2 < 4 (a_1+a_2)^2$ is satisfied in view of
$\left(1-2(a_1^2+a_2^2) \right)^2 - 4(a_1+a_2)^2 \left(1-2(a_1+a_2)^2 \right) > 0$, as we just showed. \\

iii) This leaves only the minimum case for $0 \le t < \frac 1 {2 \sqrt 3}$. For $a_1 = a_2$ with $\frac 1 {2 \sqrt 3} \le a_1=a_2 \le \frac 1 2$, we have by \eqref{eq5.5} for $t = t_-(a_1)$
$$F(a_1,t_-(a_1)) = \frac 4 {(1+36a_1^2)^2} \left[ a_1(5+52a_1^2) + (2 a_1^2+ \frac 1 2) \sqrt{28a_1^2-1} \right] \ . $$
This is decreasing in $a_1$. On the other hand, $A(a^{(1)},t) = \left(\frac{\sqrt 3} 2 \right)^3 \left(\frac{\sqrt 3} 2 - t_-(a_1) \right)^2$ is increasing in $a_1$. For the maximal possible $a_1 \le \tilde{a_1} \simeq 0.4268$ (zero of $t_-$) we have $t_-(\tilde{a_1}) = 0$ and $A(a^{(1)},0) = \frac 9 {32} \sqrt 3 \simeq 0.4871 < F(a_1,0) \simeq 0.5551$. Hence $A(a^{(1)},t) < F(a_1,t)$ for all $0 \le t < \frac 1 {2 \sqrt 3}$, and $a^{(1)}$ is the minimum of $A(\cdot,t)$. \\
In the second case $t = t(a_1,a_2) = \frac{1-2 (a_1+a_2)^2 + 4a_1a_2}{4(a_1+a_2)}$, as above $A(a,t) = \frac 1 {2(a_1+a_2)}$. This is larger than
$A(a^{(1)},t) = \left(\frac{\sqrt 3} 2 \right)^3 \left(\frac{\sqrt 3} 2 -t(a_1,a_2) \right)^2$: Using $a_2 \ge \frac 1 6 \sqrt{6-8a_1^2}- \frac{a_1} 3$, $a_2 > 0$, the difference
$A(a,t)-A(a^{(1)},t)$ is minimal for $a_1=a_2= \tilde{a_1}$ which then is $0.0986 > 0$. This case does not lead to the minimum of $A(\cdot,t)$.  \hfill $\Box$

\vspace{0,5cm}

\end{document}